\documentclass[12pt]{amsart}
\usepackage{amsmath,amsfonts,amssymb,latexsym}
\usepackage{hyperref}
\usepackage{graphicx}
\usepackage[a4paper, left=2cm, right=2cm, top=2cm, bottom=2cm]{geometry}

\newcommand{\RR}{\mathbb{R}}
\newcommand{\CC}{\mathbb{C}}
\newcommand{\NN}{\mathbb{N}}
\newcommand{\ZZ}{\mathbb{Z}}
\newcommand{\QQ}{\mathbb{Q}}
\newcommand{\EE}{\mathbb{E}}
\newcommand{\OO}{\mathcal{O}}
\newcommand{\Oo}{\mathcal{O}}
\newcommand{\Ss}{\mathcal{S}}
\newcommand{\Bo}{\widehat{\mathcal{B}}}
\newcommand{\RE}{ {\rm Re \,} }

\newtheorem{Tw}{Theorem}

\newtheorem{Stw}{Proposition}

\theoremstyle{definition}
\newtheorem{Df}{Definition}

\theoremstyle{remark}
\newtheorem{Uw}{Remark}
\newtheorem{Prz}{Example}

\begin{document}
\title[The Stokes phenomenon for some moment PDEs]{The Stokes phenomenon for some moment partial differential equations}
\author{S{\l}awomir Michalik}
\address{Faculty of Mathematics and Natural Sciences,
College of Science\\
Cardinal Stefan Wyszy\'nski University\\
W\'oycickiego 1/3,
01-938 Warszawa, Poland\\
ORCiD: 0000-0003-4045-9548}
\email{s.michalik@uksw.edu.pl}
\urladdr{\url{http://www.impan.pl/~slawek}
}
\author{Bo\.{z}ena Tkacz}
\address{Faculty of Mathematics and Natural Sciences,
College of Science\\
Cardinal Stefan Wyszy\'nski University\\
W\'oycickiego 1/3,
01-938 Warszawa, Poland}
\email{bpodhajecka@o2.pl}

\begin{abstract}
We study the Stokes phenomenon for the solutions of general homogeneous linear moment partial differential 
equations with constant coefficients in two complex variables under condition that the Cauchy data are holomorphic
on the complex plane but finitely many singular or branching points with the appropriate growth condition at the infinity.
The main tools are the theory of summability and multisummability, and the theory of hyperfunctions. Using them we describe Stokes lines,
anti-Stokes lines, jumps across Stokes lines, and a maximal family of solutions.
\end{abstract}

\maketitle

\section{Introduction}
In this article, we generalise our results from \cite {Mic-Pod} concerning summability and Stokes phenomenon for the formal solutions
of the Cauchy problem for the complex heat equation. In the present paper, we consider the Cauchy problem for 
general homogeneous linear moment partial differential equation with constant coefficients in two complex variables $(t,z)$
\begin{equation}
\label{eq:1}
\begin{cases}
P(\partial_{m_1,t},\partial_{m_2,z})u=0&\\
\partial_{m_1,t}^j u(0,z)=\varphi_j(z)\in\OO(D),&j=0,\dots,N-1,
\end{cases}
\end{equation}
where $P(\lambda,\zeta)$ is a polynomial of two variables of degree $N$ with respect to $\lambda$. Here $\partial_{m_1,t}$ and
$\partial_{m_2,z}$ denote the formal moment differentiations introduced by W.~Balser and M.~Yoshino \cite{B-Y}, which generalise the usual and
fractional differentiations.

Such type of equations was previously investigated by the first author \cite{Mic7,Mic8,Mic11} and by A.~Lastra, S.~Malek and J.~Sanz  \cite{La-Ma-Sa},
mainly in the context of multisummability in a given direction. 

Now we use the similar methods as in the above mentioned papers to
the study of multisummable normalised formal solution $\widehat{u}$ of (\ref{eq:1}). 
It means that $\widehat{u}$ has to be multisummable in every direction but finitely many singular directions.  
For this reason we assume that
the Cauchy data have finitely many singular or branching points $z_0,\dots,z_n\in\CC\setminus\{0\}$ and are analytically continued to
$\CC\setminus\bigcup_{j=0}^n\{z_jt\colon t\geq 1\}$, and that satisfy the appropriate exponential growth condition at the infinity.
Observe that by the linearity of (\ref{eq:1}) it is sufficient to consider the case when there is exactly one such point,
say $z_0\in\CC\setminus\{0\}$. Therefore we only consider the case when $\varphi_j(z)\in\OO(\widetilde{\CC\setminus\{z_0\}})$.

Using such formal multisummable solution $\widehat{u}$, for any nonsingular admissible multidirection $\mathbf{d}$ we are able to construct its
multisum $u^{\mathbf{d}}$. This multisum is an actual solution of (\ref{eq:1}) as a holomorphic function in some sectorial
neighbourhood of the origin.

The main purpose of this article is the description of these
actual solutions and the study of the relations between them. To this end we introduce the concept of maximal family of solutions.
It is defined as the whole family of actual solutions, which can be obtained by the method of multisummability. 

The relations between solutions are studied in the context of
the Stokes phenomenon. It means that we find the Stokes lines, which separate different actual solutions constructed from 
the same multisummable formal power series solution.
We also calculate the differences between actual solutions on such lines, which are called jumps across the Stokes lines.
To study such jumps we apply the Laplace type hyperfunctions supported on the Stokes line.

In this way we get the main result of the paper about the maximal family of solutions and the Stokes phenomenon for (\ref{eq:1}), which is
given in Theorem \ref{th:multisummability}.

In the special case when $\partial_{m_1,t}$ and $\partial_{m_2,z}$ are replaced by $\partial_{t}$ and $\partial_{z}$ we get the description
of the Stokes phenomenon for general linear PDEs with constant coefficients.

In this sense the paper gives the application of theory of summability for PDEs to the description of
maximal family of solutions and to the study of Stokes phenomenon for such equations.

Let us recall that the theory of summability of the formal solutions of PDEs has been recently intensively developed by such authors as
M.~Hibino \cite{H}, K.~Ichinobe and M.~Miyake \cite{I-Miy},
K.~Ichinobe \cite{I3,I4}, A.~Lastra, S.~Malek and J.~Sanz \cite{La-Ma-Sa2}, P.~Remy \cite{Re}, H.~Tahara and H.~Yamazawa \cite{T-Y},
H.~Yamazawa and M.~Yoshino \cite{Ya-Yo},
M.~Yoshino \cite{Yo,Yo2},  and others.

The paper is organized as follows. Section 2 consists of basic notations. In Section 3 we recall Balser's theory of moment summability.
In particular, we introduce kernel functions and connected with them moment functions, Gevrey order, moment Borel and Laplace transforms,
$k$-summability and multisummability. In the next section we recall the concept of moment differential operators and their generalisation
to pseudodifferential operators. In Section 5 we recall the notion of Stokes phenomenon. We define Stokes lines and jumps across them
for multisummable formal power series. We also introduce Laplace type hyperfunction on Stokes lines, which allows us to describe
these jumps. In Section 6 we introduce the idea of a maximal family of normalised actual solutions of non-Kowalevskian equation. 
We describe such family of solutions of (\ref{eq:1}) in the case when formal solution $\widehat{u}$ is multisummable (Theorem \ref{th:general}).
In Section 7 we recall how to reduce the Cauchy problem (\ref{eq:1}) to a family of the Cauchy problems of simple pseudodifferential equations.
Next, using the theory of moment summability, we find the integral representation of actual solutions of these simple pseudodifferential equations
in the case when their formal solutions are summable (Proposition \ref{pr:sum}). It allows us to describe a maximal family of solutions
of simple equations, Stokes lines, and jumps across them (Theorem \ref{th:summability}). Finally we return to the equation (\ref{eq:1})
and using the theory of multisummability we get the main result of the paper, i.e. the description of a maximal family of solution,
Stokes lines and jumps across them for the equation (\ref{eq:1}), which is given in Theorem \ref{th:multisummability}. 
In the last section we present a few examples of special cases of moment partial differential equations with constant coefficients, where by using hyperfunctions we derive the form of jumps across obtained Stokes lines.

\section{Notation}
A~\emph{sector $S$ in a direction $d\in\RR$ with an opening $\alpha>0$ and a radius $R\in\RR_+$} in the universal covering space $\tilde{\CC}$ of
$\CC\setminus\{0\}$ is defined by
\begin{displaymath}
S=S_d(\alpha,R)=\{z\in\tilde{\CC}\colon\ z=r\*e^{i\phi},\ r\in(0,R),\ \phi\in(d-\alpha/2,d+\alpha/2)\}.
\end{displaymath}
 This sector is called \emph{unbounded} if $R=+\infty$ and the notation $S=S_d(\alpha)$ will be used. If the opening $\alpha$ is not essential,
 the sector $S_d(\alpha)$ is denoted briefly by $S_d$.
 
 A complex disc $D_r$ in $\CC$ with a radius $r>0$ is a set of the form 
 \begin{displaymath}
 D_r=\{z\in\CC:|z|<r\}.
 \end{displaymath}
 In case that the radius $r$ is not essential, the set $D_r$ will be designated  briefly by $D$. We also denote briefly a \emph{disc-sector}
 $S_d(\alpha)\cup D$
 (resp. $S_d\cup D$) by $\widehat{S}_d(\alpha)$ (resp. $\widehat{S}_d$).
 
 If a function $f$ is holomorphic on a domain $G\subset\CC^n$, then it will be denoted by $f\in\OO(G)$.
 Analogously, the space of holomorphic functions of the variable $z^{1/\gamma}=(z_1^{1/\gamma_1},\dots,z_n^{1/\gamma_n})$
 on a domain $G\subset\CC^n$ is denoted by $\OO_{1/\gamma}(G)$, where $z=(z_1,\dots,z_n)\in\CC^n$, $\gamma=(\gamma_1,\dots,\gamma_n)\in\NN^n$
 and $1/\gamma=(1/\gamma_1,\dots,1/\gamma_n)$.
 In other words $f\in\OO_{1/\gamma}(G)$ if and only if the function $w\mapsto f(w^{\gamma})$ is analytic for every
 $w^{\gamma}=(w_1^{\gamma_1},\dots,w_n^{\gamma_n})\in G$.
 
 More generally, if $\EE$ denotes a complex Banach space with a norm $\|\cdot\|_{\EE}$, then
by $\Oo(G,\EE)$ (resp. $\Oo_{1/\gamma}(G,\EE)$) we shall denote the set of all $\EE$-valued  
holomorphic functions (resp. holomorphic functions of the variables $z^{1/\gamma}$)
on a domain $G\subseteq\CC^n$.
For more information about functions with values in Banach spaces we refer the reader to \cite[Appendix B]{B2}. 
In the paper, as a Banach space $\EE$ we will
take the space of complex numbers $\CC$ (we abbreviate $\Oo(G,\CC)$ to $\Oo(G)$ and
$\Oo_{1/\gamma}(G,\CC)$ to $\Oo_{1/\gamma}(G)$)
or the space of functions $E_{1/\gamma}(D):=\Oo_{1/\gamma}(D)\cap C(\overline{D})$ equipped with the norm
$\|\varphi\|_{E_{1/\gamma}(D)}:=\max_{z\in\overline{D}}|\varphi(z)|$.
 
The space of formal power series $\sum_{n=0}^{\infty} a_{n} t^{n}$ with $a_n\in\EE$ is denoted by $\EE[[t]]$.

We use the ``hat'' notation ($\widehat{u}$, $\widehat{u}_i$, $\widehat{f}$) to denote the formal power series. If the formal power series
$\widehat{u}$ (resp. $\widehat{u}_i$, $\widehat{f}$) is convergent, we denote its sum by $u$ (resp. $u_i$, $f$).

\begin{Df}
Suppose $k\in\RR$, $S$ is an unbounded sector and $u\in\OO_{1/\gamma}(S,\EE)$. The function $u$ is of \emph{exponential growth of order at most $k$},
if for every proper subsector $S^*\prec S$ (i.e. $\overline{S^{*}}\setminus\{0\} \subseteq S$) there exist constants $C_1, C_2>0$ such that
$\|u(x)\|_{\EE}\le C_1\*e^{C_2|x|^{k}}$ for every $x\in S^*$. If this is so, one can write $u\in\OO_{1/\gamma}^{k}(S,\EE)$
and $u\in\OO_{1/\gamma}^k(\CC,\EE)$ for $S=\CC$.
\par
More generally, if $G$ is an unbounded domain in $\CC^n$ and $u\in\OO_{1/\gamma}(G,\EE)$, then $u\in\OO_{1/\gamma}^{k}(G,\EE)$ if
for every set $G^*$ satisfying $\overline{G^*}\subset\mathrm{Int}\,G$ there exist constants $C_1, C_2>0$ such that $\|u(x)\|_{\EE}\le C_1\*e^{C_2|x|^{k}}$ for every $x\in G^*$.
\end{Df}

\section{Kernel and moment functions, k-summability and multisummability}
In this section we recall the notion of moment methods introduced by Balser \cite{B2}.
It allows us to describe moment Borel transforms, Gevrey order, Borel summability and multisummability
\begin{Df}[see {\cite[Section 5.5]{B2}}]
    \label{df:moment}
    A pair of functions $e_m$ and $E_m$ is said to be \emph{kernel functions of order $k$} ($k>1/2$) if
    they have the following properties:
   \begin{enumerate}
    \item[1.] $e_m\in\Oo(S_0(\pi/k))$, $e_m(z)/z$ is integrable at the origin, $e_m(x)\in\RR_+$ for $x\in\RR_+$ and
     $e_m$ is exponentially flat of order $k$ as $z\to\infty$ in $S_0(\pi/k)$ (i.e. for every $\varepsilon > 0$ there exist ${A,B > 0}$
     such that $|e_m(z)|\leq A e^{-(|z|/B)^k}$ for $z\in S_0(\pi/k-\varepsilon)$).
    \item[2.] $E_m\in\Oo^{k}(\CC)$ and $E_m(1/z)/z$ is integrable at the origin in $S_{\pi}(2\pi-\pi/k)$.
    \item[3.] The connection between $e_m$ and $E_m$ is given by the \emph{corresponding moment function
    $m$ of order $1/k$} as follows.
     The function $m$ is defined by the Mellin transform of $e_m$
     \begin{gather}
      \label{eq:e_m}
      m(u):=\int_0^{\infty}x^{u-1} e_m(x)dx \quad \textrm{for} \quad \RE u \geq 0
     \end{gather}
     and the kernel function $E_m$ has the power series expansion
     \begin{gather}
      \label{eq:E_m}
      E_m(z)=\sum_{n=0}^{\infty}\frac{z^n}{m(n)} \quad  \textrm{for} \quad z\in\CC.
     \end{gather}
    \item[4.] Additionally we assume that the corresponding moment function satisfies the normalisation property $m(0)=1$.
   \end{enumerate}
   \end{Df}
   
   \begin{Uw}
    Observe that by the inverse Mellin transform and by (\ref{eq:E_m}), the moment function $m$
    uniquely determines the kernel functions $e_m$ and $E_m$.
   \end{Uw}

    In case $k\leq 1/2$ the set $S_{\pi}(2\pi-\pi/k)$ is not defined,
    so the second property in Definition \ref{df:moment} can not be satisfied. It means that we
    must define the kernel functions of order $k\leq 1/2$ and the corresponding moment functions
    in another way. To this end we use the ramification at $z=0$.
    
    \begin{Df}[see {\cite[Section 5.6]{B2}}]
     \label{df:small}
     A function $e_m$ is called a \emph{kernel function of order $k>0$} if we
     can find a pair of kernel functions $e_{\widetilde{m}}$ and $E_{\widetilde{m}}$ of
     order $pk>1/2$ (for some $p\in\NN$) so that
     \begin{gather*}
      e_m(z)=e_{\widetilde{m}}(z^{1/p})/p \quad \textrm{for} \quad z\in S_0(\pi/k).
     \end{gather*}
     For a given kernel function $e_m$ of order $k>0$ we define the
     \emph{corresponding moment function $m$ of order $1/k>0$} by (\ref{eq:e_m}) and
     the \emph{kernel function $E_m$ of order $k>0$} by (\ref{eq:E_m}).
    \end{Df}
    
    \begin{Uw}
     Observe that by Definitions \ref{df:moment} and \ref{df:small} we have
     \begin{eqnarray*}
      m(u)=\widetilde{m}(pu) & \textrm{and} &
      E_m(z)=\sum_{j=0}^{\infty}\frac{z^j}{m(j)}=\sum_{j=0}^{\infty}\frac{z^j}{\widetilde{m}(jp)}.
     \end{eqnarray*}
    \end{Uw}

As in \cite{Mic8}, we extend the notion of moment functions to real orders.
\begin{Df}
 \label{df:moment_general}
     We say that $m$ is a \emph{moment function of order $1/k<0$} if $1/m$ is a moment function of order $-1/k>0$.
     \par
     We say that $m$ is a \emph{moment function of order $0$} if there exist moment functions $m_1$ and $m_2$ of the same order $1/k>0$ such that $m=m_1/m_2$.
\end{Df}

By Definition \ref{df:moment_general} and by \cite[Theorems 31 and 32]{B2} we have
\begin{Stw}
 Let $m_1$, $m_2$ be moment functions of orders $s_1,s_2\in\RR$ respectively. Then
 \begin{itemize}
  \item $m_1m_2$ is a moment function of order $s_1+s_2$,
  \item $m_1/m_2$ is a moment function of order $s_1-s_2$.
 \end{itemize}
\end{Stw}

\begin{Prz}
\label{ex:functions}
 For any $k>0$ the classical kernel functions and the corresponding moment function,
 satisfying Definition \ref{df:moment} or \ref{df:small}, are given by
 \begin{itemize}
  \item $e_m(z)=kz^{k}e^{-z^k}$,
  \item $m(u)=\Gamma(1+u/k)$,
  \item $E_m(z)=\sum_{j=0}^{\infty}\frac{z^j}{\Gamma(1+j/k)}=:\mathbf{E}_{1/k}(z)$, where $\mathbf{E}_{1/k}$ is the
    Mittag-Leffler function of index $1/k$.
\end{itemize}
They are used in the classical theory of $k$-summability.
\end{Prz}

\begin{Prz}
For any $s\in\RR$ we will denote by $\Gamma_s$ the function
\[
  \Gamma_s(u):=\left\{
  \begin{array}{lll}
    \Gamma(1+su) & \textrm{for} & s \geq 0\\
    1/\Gamma(1-su) & \textrm{for} & s < 0.
  \end{array}
  \right.
\]
Observe that by Example \ref{ex:functions} and Definition \ref{df:moment_general}, $\Gamma_s$
is an example of a moment function of order $s\in\RR$.
\end{Prz}

The moment functions $\Gamma_s$ will be extensively used in the paper,
since every moment function $m$ of order $s$ has the same growth as $\Gamma_s$. Precisely speaking,
we have 
\begin{Stw}[see {\cite[Section 5.5]{B2}}]
  \label{pr:order}
  If $m$ is a moment function of order $s\in\RR$
  then there exist constants $a,A,c,C>0$ such that
    \begin{gather*}
    a c^n\Gamma_s(n)\leq m(n) \leq A C^n\Gamma_s(n) \quad \textrm{for every} \quad n\in\NN_0.
    \end{gather*}
  \end{Stw}
  
Using Balser's theory of general moment summability (\cite[Section 6.5]{B2}, in particular \cite[Theorem 38]{B2}), 
we apply the moment functions to define moment Borel transforms, the Gevrey order and the Borel summability. We first introduce

\begin{Df}
 Let $m$ be a moment function. Then the linear operator $\Bo_{m}\colon \EE[[t]]\to\EE[[t]]$ defined by
 \[
  \Bo_{m}\big(\sum_{j=0}^{\infty}u_jt^{j}\big):=
  \sum_{j=0}^{\infty}\frac{u_j}{m(j)}t^{j}
 \]
 is called an \emph{$m$-moment Borel transform}.
\end{Df}

We define the Gevrey order of formal power series as follows
   \begin{Df}
    \label{df:summab}
    Let $s\in\RR$. Then
    $\widehat{u}\in\EE[[t]]$ is called a \emph{formal power series of Gevrey order $s$} if
    there exists a disc $D\subset\CC$ with centre at the origin such that
    $\Bo_{\Gamma_s}\widehat{u}\in\Oo(D,\EE)$. The space of formal power series of Gevrey 
    order $s$ is denoted by $\EE[[t]]_s$.
   \end{Df}

\begin{Uw}
 By Proposition \ref{pr:order}, we may replace $\Gamma_s$ in Definition \ref{df:summab} by any moment function $m$ of the same order $s$.
\end{Uw}

\begin{Uw}
 If $\widehat{u}\in\EE[[t]]_s$ and $s\leq 0$ then the formal series $\widehat{u}$ is convergent,
 so its sum $u$ is well defined.
 Moreover, $\widehat{u}\in\EE[[t]]_0 \Longleftrightarrow u\in\Oo(D,\EE)$ and
 $\widehat{u}\in\EE[[t]]_s \Longleftrightarrow u\in\Oo^{-1/s}(\CC,\EE)$ for $s<0$.
\end{Uw}

\begin{Df}
  Let $e_m, E_m$ be a pair of kernel functions of order $1/k>0$ with a moment function $m$ and let $d\in\RR$.
  \begin{itemize}
  \item If $v\in\Oo^{k}(\widehat{S}_d,\EE)$ then the integral operator $T_{m,d}$ defined by
  \[
   (T_{m,d}v)(t):=\int_{e^{id}\RR_+}e_{m}(s/t)v(s)\frac{ds}{s}
  \]
 is called an \emph{$m$-moment Laplace transform in a direction $d$}.
 \item If $v\in\Oo(S_d(\frac{\pi}{k}+\varepsilon,R),\EE)$ for some $\varepsilon,R>0$ then the integral operator $T^-_{m,d}$ defined by
  \[
   (T^-_{m,d}v)(s):=-\frac{1}{2\pi i}\int_{\gamma(d)}E_{m}(s/t)v(t)\frac{dt}{t}
  \]
(where a path $\gamma(d)$ is the boundary of a sector contained in $S_d(\frac{\pi}{k}+\varepsilon,R)$ with bisecting direction $d$,
a finite radius, an opening slightly larger than $\pi/k$, and the orientation is negative) is called an \emph{inverse $m$-moment Laplace transform in a direction $d$}.
 \end{itemize}
\end{Df}

\begin{Uw}
 Observe, that $T_{m,d}(t^{n})=m(n)t^{n}$ for every $n\in\NN_0$. Hence $ T_{m,d}\Bo_{m}u=u$
 for every $u\in\Oo(D)$.
\end{Uw}

Now we are ready to define the summability of formal power series
\begin{Df}
Let $k>0$ and $d\in\RR$. Then $\widehat{u}\in\EE[[t]]$ is called
\emph{$k$-summable in a direction $d$} if there exist $\varepsilon>0$ and a disc-sector $\widehat{S}_d=\widehat{S}_d(\varepsilon)$
in a direction $d$ such that
$v=\Bo_{\Gamma_{1/k}}\widehat{u}\in\Oo^k(\widehat{S}_d,\EE)$.
\par
 Moreover, the \emph{$k$-sum of $\widehat{u}$ in the direction $d$} is given by
\begin{equation}
\label{eq:sum}
       u^d(t)=\mathcal{S}_{k,d}\widehat{u}(t):=(T_{m,\theta}v)(t)=\int_{e^{i\theta}\RR_+}e_{m}(s/t)v(s)\frac{ds}{s}\quad\textrm{for}\quad
       \theta\in(d-\varepsilon/2,d+\varepsilon/2).
\end{equation}
\end{Df}

\begin{Df}
If $\widehat{u}\in\EE[[t]]$ is $k$-summable in all directions $d$ but (after identification modulo $2\pi$)
finitely many directions $d_1,\dots,d_n$ then
$\widehat{u}$ is called \emph{$k$-summable} and $d_1,\dots,d_n$ are called \emph{singular directions of $\widehat{u}$}.
\end{Df}

Next we extend the notion of $k$-summable formal power series to that which are multisummable.

\begin{Df}
\label{df:multisummable}
 Let $k_1>\dots>k_n>0$ and let $\kappa_1,\dots,\kappa_n$ be defined by $\kappa_1=k_1$, $1/\kappa_j=1/k_j-1/k_{j-1}$, $2\leq j \leq n$.
 We say that a real vector $\mathbf{d}=(d_1,\dots,d_n)$ is an \emph{admissible multidirection with respect to} $\mathbf{k}=(k_1,\dots,k_n)$ if
 \begin{eqnarray*}
  2\kappa_j|d_j-d_{j-1}| \leq \pi & \textrm{for} & j=2,\dots,n.
 \end{eqnarray*}
\end{Df}

\begin{Uw}
 Admissibility of $\mathbf{d}$ with respect to $\mathbf{k}$ is equivalent to the inclusions $I_1\subseteq I_2\subseteq \dots\subseteq I_n$,
 where $I_j:=(d_j-\frac{\pi}{2k_j},d_j+\frac{\pi}{2k_j})$ for $j=1,\dots,n$.
\end{Uw}

\begin{Df}
  Let $m_1,\dots,m_n$ be moment functions of positive orders respectively $1/\kappa_1,\dots,1/\kappa_n$, where $\kappa_1,\dots,\kappa_n$ are
  constructed in Definition \ref{df:multisummable}.
  A formal power series $\widehat{u}(t)=\sum_{j=0}^{\infty}u_jt^{j}\in\EE[[t]]$ is called
  \emph{$\mathbf{k}$-multisummable in the admissible multidirection
  $\mathbf{d}$}, provided that
  \begin{itemize}
  \item $v_n(t):=(\Bo_{m_n}\cdots\Bo_{m_1}\widehat{u})(t)
  =\sum\limits_{j=0}^{\infty}\frac{u_j}{m_1(j)\cdots m_n(j)}t^{j}\in \Oo^{\kappa_n}(\widehat{S}_{d_n})$,
  \item $v_{j-1}(t):=(T_{m_j,d_j}v_j)(t)\in\Oo^{\kappa_{j-1}}(S_{d_{j-1}})$ for $j=n,n-1,\dots,2$.
  \end{itemize}
  \par
  Moreover, the \emph{$\mathbf{k}$-multisum of $\widehat{u}$ in the multidirection $\mathbf{d}$} is given by 
 $$
u^{\mathbf{d}}(t)=\mathcal{S}_{\mathbf{k},\mathbf{d}}\widehat{u}(t):=(T_{m_1,d_1}\cdots T_{{m_n},d_n}v_n)(t).
 $$
\end{Df}
 \par
 \begin{Df}
 If $(d_1,\dots,d_n)$ is an admissible multidirection and the functions $v_n,\dots,v_j$ all exist, but
 $v_j\not\in\Oo^{\kappa_{j}}(S_{d_{j}})$ then $d_j$ is called a \emph{singular direction of $\widehat{u}$ of level $k_j$}
 (for $j=1,\dots,n$).
\end{Df}
\begin{Df}
 If $\widehat{u}$ has at most (after identification modulo $2\pi$) finitely many singular directions of each level $k_j$, $1\leq j \leq n$,
 then $\widehat{u}$ is called \emph{$\mathbf{k}$-multisummable}.
\end{Df}

\begin{Uw}
 \label{re:suma}
 If $k_1>\dots>k_n>0$, $(d_1,\dots,d_n)$ is an admissible multidirection and  $\widehat{u}_j$ is $k_j$-summable
 in a direction $d_j$ for $j=1,\dots,n$, then, by \cite[Lemma 20]{B2},
 $\widehat{u}:=\widehat{u}_1+\cdots+\widehat{u}_n$ is $\mathbf{k}$-multisummable in the multidirection $\mathbf{d}$
 and $\Ss_{\mathbf{k},\mathbf{d}}\widehat{u}(t)=\Ss_{k_1,d_1}\widehat{u}_1(t)+\cdots+\Ss_{k_n,d_n}\widehat{u}_n(t)$.

 Moreover, if additionally $\widehat{u}_j$ is $k_j$-summable with $n_j$ singular directions
 $d_{j,1},\dots,d_{j,n_j}$ (for $j=1,\dots,n$)
 then $\widehat{u}$ is $\mathbf{k}$-multisummable and $d_{j,1},\dots,d_{j,n_j}$ are singular directions of $\widehat{u}$ of level $k_j$.
\end{Uw}

\section{Moment operators}
     In this section we recall the notion of moment differential operators constructed by Balser and Yoshino
     \cite{B-Y} and the concept of moment pseudodifferential operators introduced in the previous
     papers of the first author \cite{Mic7,Mic8}.
     
   \begin{Df}
    Let $m$ be a moment function. Then the linear operator
    $\partial_{m_,x}\colon\EE[[x]]\to\EE[[x]]$
    defined by
    \[
     \partial_{m,x}\Big(\sum_{j=0}^{\infty}\frac{u_{j}}{m(j)}x^{j}\Big):=
     \sum_{j=0}^{\infty}\frac{u_{j+1}}{m(j)}x^{j}
    \]
    is called the \emph{$m$-moment differential operator $\partial_{m,x}$}.
    \end{Df}
    
     Below we present most important examples of moment differential operators.
     Other examples, including also integro-differential operators, can be found in
     \cite[Example 3]{Mic8}.
     \begin{Prz} 
     If $m(u)=\Gamma_1(u)$ then the operator $\partial_{m,x}$ coincides with the usual differentiation $\partial_x$.
     More generally, if $s>0$ and $m(u)=\Gamma_s(u)$ then the operator $\partial_{m,x}$ satisfies
     $(\partial_{m,x}\widehat{u})(x^s)=\partial^s_x(\widehat{u}(x^s))$,
     where $\partial^s_x$ denotes the Caputo fractional derivative of order $s$
     defined by
       $$
     \partial^{s}_{x}\Big(\sum_{j=0}^{\infty}\frac{u_{j}}{\Gamma_s(j)}x^{sj}\Big):=
     \sum_{j=0}^{\infty}\frac{u_{j+1}}{\Gamma_s(j)}x^{sj}.$$
     \end{Prz}
     \par
Immediately by the definition, we obtain the following connection between the moment Borel transform and the moment
differentiation.
\begin{Stw}
\label{pr:commutation}
Let $m$ and $m'$ be two moment functions. Then
the operators $\Bo_{m'}, \partial_{m,t}\colon\EE[[t]]\to\EE[[t]]$ satisfy the following commutation formulas
for every $\widehat{u}\in\EE[[t]]$ and for $\overline{m}=mm'$:
\begin{enumerate}
\item[i)] $\Bo_{m'}\partial_{m,t}\widehat{u}=\partial_{\overline{m},t}\Bo_{m'}\widehat{u}$,
\item[ii)] $\Bo_{m'}P(\partial_{m,t})\widehat{u}=P(\partial_{\overline{m},t})\Bo_{m'}\widehat{u}$
for any polynomial $P$ with constant coefficients.
\end{enumerate}
\end{Stw}

Now, following \cite{Mic8} we generalise moment differential operators to a kind of pseudodifferential operators. Namely, we have
\begin{Df}[{\cite[Definition 13]{Mic8}}]
Let $m$ be a moment function of order $1/k>0$ and
 $\lambda(\zeta)$ be an analytic function of the variable
$\xi=\zeta^{1/\gamma}$ for $|\zeta|\geq r_0$
(for some $\gamma\in\NN$ and $r_0>0$) of polynomial growth at infinity.
A \emph{moment pseudodifferential operator}
 $\lambda(\partial_{m,z})\colon\Oo_{1/\gamma}(D)\to\Oo_{1/\gamma}(D)$ is defined by
 \begin{equation*}
  \lambda(\partial_{m,z})\varphi(z):=\frac{1}{2\gamma\pi i} \oint^{\gamma}_{|w|=\varepsilon}\varphi(w)
  \int_{r_0e^{i\theta}}^{\infty(\theta)}\lambda(\zeta)
  E_{\widetilde{m}}(\zeta^{1/\gamma} z^{1/\gamma})\frac{e_m(\zeta w)}{\zeta w}\,d\zeta\,dw
 \end{equation*}
for every $\varphi\in\Oo_{1/\gamma}(D_r)$ and $|z|<\varepsilon < r$, where $\widetilde{m}(u):=m(u/\gamma)$,
$E_{\widetilde{m}}(\zeta^{1/\gamma} z^{1/\gamma})=\sum_{n=0}^{\infty}\frac{\zeta^{n/\gamma}z^{n/\gamma}}{m(n/\gamma)}$,
$\theta\in (-\arg w-\frac{\pi}{2k}, -\arg w + \frac{\pi}{2k})$ and $\oint_{|w|=\varepsilon}^{\gamma}$ means that we integrate 
$\gamma$ times along the positively oriented circle of radius $\varepsilon$. Here the integration in the inner integral is taken over
a ray $\{re^{i\theta}\colon r\geq r_0\}$.
\end{Df}

\begin{Df}[{\cite[Definition 9]{Mic7}}]
Let $\lambda(\zeta)$ be an analytic function of the variable
$\xi=\zeta^{1/\gamma}$ for $|\zeta|\geq r_0$
(for some $\gamma\in\NN$ and $r_0>0$) of polynomial growth at infinity. Then we define \emph{the pole order $q\in\QQ$}
and \emph{the leading term $\lambda_0\in\CC\setminus\{0\}$}
   of $\lambda(\zeta)$ as the numbers satisfying the formula
   $\lim_{\zeta\to\infty}\lambda(\zeta)/\zeta^q=\lambda_0$.
   We write it also $\lambda(\zeta)\sim\lambda_0\zeta^q$.
\end{Df}

\section{Stokes phenomenon and hyperfunctions}
Now we extend the concept of the Stokes phenomenon (see \cite[Definition 7]{Mic-Pod}) to multisummable formal power series
$\widehat{u}\in\EE[[t]]$.
\begin{Df}
\label{df:stokes}
Assume that $\widehat{u}\in\EE[[t]]$ is $\mathbf{k}$-multisummable with singular directions
$d_{j,1},\dots,d_{j,n_j}$ of level $k_j$, $1\leq j \leq n$. Then for every $l=1,\dots,n_j$ and $j=1,\dots,n$ the set
$\mathcal{L}_{d_{j,l}}=\{t\in\tilde{\CC}\colon \arg t=d_{j,l}\}$ 
is called a \emph{Stokes line of level $k_j$} for $\widehat{u}$.

Assume now that for fixed $j\in\{1,\dots,n\}$ the vector $\mathbf{d}=(d_1,\dots,d_n)$ is an admissible multidirection
with a singular direction $d_j$ of level $k_j$ and with nonsingular directions $d_l$ of level $k_l$ for $l\neq j$, 
and let $\mathbf{d_j^{\pm}}:=(d_1,\dots,d_j^{\pm},\dots,d_n)$ be the admissible multidirections, where $d_j^+$ (resp.
$d_j^-$) denotes a direction close to $d_j$ and greater (resp. less) than $d_j$,
and let $u^{\mathbf{d_j^+}}:=\Ss_{\mathbf{k},\mathbf{d_j^+}}\widehat{u}$ (resp. $u^{\mathbf{d_j^-}}:=\Ss_{\mathbf{k},\mathbf{d_j^-}}\widehat{u}$)
then the difference $J_{\mathcal{L}_{d_j},k_j}\widehat{u}:= u^{\mathbf{d_j^+}}-u^{\mathbf{d_j^-}}$ is called a \emph{jump for $\widehat{u}$ across
the Stokes line $\mathcal{L}_{d_j}$ of level $k_j$}.
\end{Df}

\begin{Uw}
 Every Stokes line $\mathcal{L}_{d_j}$ of level $k_j$ for $\widehat{u}$ determines also so called \emph{anti-Stokes lines
 $\mathcal{L}_{d_j\pm\frac{\pi}{2k_j}}$ of level $k_j$} for $\widehat{u}$.
\end{Uw} 

We will describe jumps across the Stokes lines in terms of hyperfunctions. The similar approach to the Stokes phenomenon one can find
in \cite{Im, Mal3, S-S}. For more information about the theory of hyperfunctions we refer the reader to \cite{Kaneko}.

We will consider the space 
$$\mathcal{H}^k(\mathcal{L}_d):=\OO^k(D\cup (S_d\setminus \mathcal{L}_d))\Big/\OO^k(\widehat{S}_d)$$
of Laplace type hyperfunctions supported by $\mathcal{L}_d$ with exponential growth of order $k$. It means that every hyperfunction
$G\in\mathcal{H}^k(\mathcal{L}_d)$ may be written as
\[
G(s)=[g(s)]_{d}=\{g(s)+h(s)\colon h(s)\in\OO^{k}(\widehat{S}_d)\}
\]
for some defining function $g(s)\in\OO^k(D\cup (S_d\setminus \mathcal{L}_d))$. 

Let $\gamma_{d}$ be a path consisting of the half-lines from 
$e^{id^-}\infty$ to $0$ and from $0$ to $e^{id^+}\infty$, i.e.
$\gamma_{d}=-\gamma_{d^-}+\gamma_{d^+}$ with $\gamma_{d^{\pm}}=\mathcal{L}_{d^{\pm}}$.
By the K\"othe type theorem \cite{Kot} one can treat the hyperfunction $G(s)=[g(s)]_d$ as the analytic functional defined by
\begin{gather}
\label{eq:kothe}
G(s)[\varphi(s)]:=\int_{\gamma_d}g(s)\varphi(s)\,ds,
\end{gather}
for such small $\varphi\in\OO^{-k}(\widehat{S}_d)$
that the function $s\mapsto g(s)\varphi(s)$ belongs to the space
$\OO^{-k}(D\cup (S_d\setminus \mathcal{L}_d))$.

To describe the jumps across the Stokes lines in terms of hyperfunctions,
first assume that $\widehat{f}\in\CC[[t]]$ is $k$-summable, $m$ is a moment function of order $1/k$ and $d$ is a singular direction. 
By (\ref{eq:sum}) the jump for $\widehat{f}$ across the Stokes line $\mathcal{L}_d$ is given by 
\begin{equation}
\label{eq:jump_hyp}
 J_{\mathcal{L}_d}\widehat{f}(t)=f^{d^+}(t)-f^{d^-}(t)=(T_{m,d^+}-T_{m,d^-})\Bo_{m}\widehat{f}(t).
\end{equation}

Observe that we can treat
$g_0(t):=\Bo_{m}\widehat{f}(t)\in\Oo^{k}(D\cup (S_{d}\setminus \mathcal{L}_{d}))$
as a defining function of the hyperfunction $G_0(s):=[g_0(s)]_{d}\in \mathcal{H}^{k}(\mathcal{L}_{d})$.
So, combining (\ref{eq:kothe}) with (\ref{eq:jump_hyp}) we conclude that 
\begin{gather}
\label{eq:J}
 J_{\mathcal{L}_d}\widehat{f}(t)=G_0(s)\Big[\frac{e_{m}(s/t)}{s}\Big]\quad \textrm{for sufficiently small}\ r>0\ \textrm{and}\
 t\in S_{d}(\frac{\pi}{k}, r).
\end{gather}

Moreover, it is natural to define the $m$-moment Laplace operator
$T_{m,d}$ acting on the hyperfunction $G(s)$ as
$T_{m,d}G(t):=G(s)\Big[\frac{e_{m}(s/t)}{s}\Big]$ for
$t\in S_{d}(\frac{\pi}{k}, r)$, where $G(s)[\varphi(s)]$ is defined by (\ref{eq:kothe}). So, by (\ref{eq:J})
we may describe the jump in terms of the $m$-moment Laplace operator acting on the hyperfunction as
$J_{\mathcal{L}_d}\widehat{f}(t)=T_{m,d}G_0(t)$.

Now, let $\widehat{f}\in\CC[[t]]$ be $\mathbf{k}$-multisummable and $\mathbf{d}$ be as in Definition \ref{df:stokes} with $\mathcal{L}_{d_j}$
being the Stokes line of level $k_j$. We additionally assume as in Remark \ref{re:suma} that $\widehat{f}=\widehat{f}_1+\cdots+\widehat{f}_n$, where
$\widehat{f}_j$ is $k_j$-summable. 
Then, by Remark \ref{re:suma}, analogously as in the summable case, the jump across $\mathcal{L}_{d_j}$ of level $k_j$ is given by
$$
J_{\mathcal{L}_{d_j},k_j}\widehat{f}=
f^{\mathbf{d_j^+}}-f^{\mathbf{d_j^-}}=f_j^{d_j^+}-f_j^{d_j^-}
=(T_{m_j,d_j^+}-T_{m_j,d_j^-})\Bo_{m_j}\widehat{f}_j
$$
and we may describe this jump in terms of hyperfunctions as in the previous case.

Similarly, if $\mathcal{L}_d$ is a Stokes line for $k$-summable $\widehat{u}=\widehat{u}(t,z)\in\Oo(D)[[t]]$, then
we are able to describe jumps for $\widehat{u}(t,z)$ at the point $z=0$ in terms of hyperfunctions. Namely we have 
\begin{gather*}
J_{\mathcal{L}_{d}}\widehat{u}(t,0)=(T_{m,d}F_{0})(t)= F_{0}(s)\Big[\frac{e_{m}(s/t)}{s}\Big],\quad \textrm{where}\quad
F_{0}(s)=[\Bo_{m}\widehat{u}(s,0)]_{d}\in \mathcal{H}^k(\mathcal{L}_{d}).
\end{gather*}

Analogously we calculate jumps across a Stokes line $\mathcal{L}_{d_j}$ of level $k_j$ for $\mathbf{k}$-multisummable $\widehat{u}$
satisfying $\widehat{u}=\widehat{u}_1+\cdots+\widehat{u}_n$, where $\widehat{u}_i$ is $k_i$-summable ($i=1,\dots,n$).

\begin{Uw}
\label{re:z}
 In some special cases we are also able to describe jumps for $\widehat{u}(t,z)$ at any point $z\in D$.
 It is possible in the case when $\widehat{u}$ is a multisummable solution of
 $$
 \begin{cases}
  P(\partial_{m_1,t},\partial_z)u=0\\
  \partial_{m_1,t}^ju(0,z)=\varphi_j(z)\quad\textrm{for}\quad j=0,\dots,N-1,
 \end{cases}
 $$
 instead of (\ref{eq:1}). In this case
 we are able to reduce the problem of description of jumps for $\widehat{u}(t,z)$ at the fixed point $z\in D$, to the problem
 of description of jumps for the auxiliary formal power series $\widehat{u}_z(t,s):=\widehat{u}(t,s+z)$ at the point $s=0$. Since the derivative
 operator
 $\partial_z$ is invariant under the translation, i.e. $(\partial_z\widehat{u})(t,s+z)=\partial_s(\widehat{u}(t,s+z))$, we conclude that
 $\widehat{u}_z(t,s)$ is a multisummable solution of
 $$
 \begin{cases}
  P(\partial_{m_1,t},\partial_s)u_z=0\\
  \partial_{m_1,t}^ju_z(0,s)=\varphi_{z,j}(s):=\varphi_j(s+z)\quad\textrm{for}\quad j=0,\dots,N-1.
 \end{cases}
 $$
 Hence $J_{\mathcal{L}_{d}}\widehat{u}(t,z)=J_{\mathcal{L}_{d_z}}\widehat{u}_z(t,0)$, where $\mathcal{L}_{d_z}$ is a Stokes line of
 $\widehat{u}_z$, which corresponds to a Stokes line $\mathcal{L}_d$ of $\widehat{u}$.
 
 Since in general the moment differential operators are not invariant under translation, we are not able to use this method to
 describe the jumps for solutions of $P(\partial_{m_1,t},\partial_{m_2,z})u=0$ at any point $z\in D$.
 \end{Uw}

\section{A maximal family of solutions}
Now we are ready to describe a family of normalised actual solutions of given non-Kowalevskian equation using sums of multisummable formal
power series solution. More precisely we consider the Cauchy problem
\begin{equation}
 \label{eq:pde}
 \begin{cases}
  P(\partial_{m_1,t},\partial_{m_2,z})u=0&\\
  \partial_{m_1,t}^j u(0,z)=\varphi_j(z)\in\OO(D),&j=0,\dots,N-1,
 \end{cases}
\end{equation}
where $m_1$, $m_2$ are moment functions of orders $s_1,s_2>0$ respectively and
\begin{equation}
 \label{eq:polyn}
P(\lambda,\zeta)=P_0(\zeta)\lambda^N-\sum_{j=1}^N P_j(\zeta)\lambda^{N-j}
\end{equation}
is a general polynomial of two variables, which is of
order $N$ with respect to $\lambda$.

If  $P_0(\zeta)$ defined by (\ref{eq:polyn}) is not a constant, then a formal solution of (\ref{eq:pde})
    is not uniquely determined. To avoid this inconvenience we choose some special solution which is already 
    uniquely determined. To this end we factorise the polynomial $P(\lambda,\zeta)$ as follows
   \begin{equation*}
    P(\lambda,\zeta)=P_0(\zeta)(\lambda-\lambda_1(\zeta))^{N_1}\cdots(\lambda-\lambda_l(\zeta))^{N_l}
    =:P_0(\zeta)\widetilde{P}(\lambda,\zeta),
    \end{equation*}
   where $\lambda_1(\zeta),\dots,\lambda_l(\zeta)$ are the roots of the characteristic equation
   $P(\lambda,\zeta)=0$ with multiplicity
   $N_1,\dots,N_l$ ($N_1+\cdots+N_l=N$) respectively.
   \par
     Since $\lambda_{\alpha}(\zeta)$ are algebraic functions,
     we may assume that there exist $\gamma\in\NN$ and $r_0<\infty$ such that
   $\lambda_{\alpha}(\zeta)$ are holomorphic functions of the variable $\xi=\zeta^{1/\gamma}$
   (for $|\zeta|\geq r_0$ and $\alpha=1,\dots,l$) and, moreover, there exist $\lambda_{\alpha}\in\CC\setminus\{0\}$ and 
   $q_{\alpha}=\mu_{\alpha}/\nu_{\alpha}$
   (for some relatively prime numbers $\mu_{\alpha}\in\ZZ$ and $\nu_{\alpha}\in\NN$) such that
   $\lambda_{\alpha}(\zeta)\sim\lambda_{\alpha}\zeta^{q_{\alpha}}$ for $\alpha=1,\dots,l$.
   Observe that $\nu_{\alpha}|\gamma$ for $\alpha=1,\dots,l$.

Hence $\lambda_{\alpha}(\partial_{m_2,z})$ are well-defined moment pseudodifferential operators
   and consequently also the operator
   $$
    \widetilde{P}(\partial_{m_1,t},\partial_{m_2,z})=(\partial_{m_1,t}-\lambda_1(\partial_{m_2,z}))^{N_1}\cdots
    (\partial_{m_1,t}-\lambda_l(\partial_{m_2,z}))^{N_l}
   $$
   is well-defined.
\par
Under the above assumption, by a \emph{normalised formal solution} $\widehat{u}$ of (\ref{eq:pde}) we mean such solution
   of (\ref{eq:pde}), which is also a solution of the pseudodifferential equation
   $\widetilde{P}(\partial_{m_1,t},\partial_{m_2,z})\widehat{u}=0$ (see \cite[Definition 10]{Mic7}).
\par
Since the principal part of the pseudodifferential operator $\widetilde{P}(\partial_{m_1,t},\partial_{m_2,z})$ with respect to $\partial_{m_1,t}$
is given by $\partial_{m_1,t}^N$,
the Cauchy problem (\ref{eq:pde}) has a unique normalised formal power series solution $\widehat{u}\in\OO(D)[[t]]$.
If we additionally assume that $\widehat{u}$ is multisummable, then using the procedure of multisummability in nonsingular directions,
we obtain a family of normalised actual solutions of (\ref{eq:pde}) on some sectors with respect to $t$.
This motivates us to introduce the following definitions.

\begin{Df}
Let $S$ be a sector in the universal covering space $\tilde{\CC}$. A function $u\in\OO(S\times D)$ is called a \emph{normalised actual solution}
of (\ref{eq:pde})
if it satisfies 
\begin{equation*}
 \begin{cases}
  \widetilde{P}(\partial_{m_1,t},\partial_{m_2,z})u=0&\\
  \lim\limits_{t\to 0,\ t\in S}\partial_{m_1,t}^j u(t,z)=\varphi_j(z)\in\OO(D),&j=0,\dots,N -1.
 \end{cases}
\end{equation*}
\end{Df}

In \cite{Mic-Pod} we introduced a maximal family of solutions of (\ref{eq:pde}) in the case when a formal power series solution is $k$-summable.
It is a collection of all actual solutions of (\ref{eq:pde}) constructed by the procedure of $k$-summability.
Now we generalise this definition to the multisummable case. 

\begin{Df}
 \label{df:maximal}
 Assume that the normalised formal power series solution $\widehat{u}$ of (\ref{eq:pde}) is $\mathbf{k}$-multisummable,
 $\mathcal{J}$ is a finite set of indices,
 and $V$ is a sector with an opening greater than $\pi/k_n$ on  the Riemann surface of $t^{\frac{1}{q}}$ for some $q\in\QQ_+$.
 \par
 We say that $\{u_i\}_{i\in \mathcal{J}}$ with $u_i\in\OO(V_i\times D)$ is a
 \emph{maximal family of solutions of (\ref{eq:pde}) on $V\times D$} if the following conditions hold:
 \begin{enumerate}
  \item[(a)] $V_i\subseteq V$ is a sector of opening greater than $\pi/k_1$ for every $i\in \mathcal{J}$.
  \item[(b)] $\{V_i\}_{i\in \mathcal{J}}$ is a covering of $V$. 
  \item[(c)] $u_i\in\OO(V_i\times D)$ is a normalised actual solution of (\ref{eq:pde}) for every $i\in \mathcal{J}$.
  \item[(d)] If $V_i\cap V_j \neq \emptyset$ then $u_i\not\equiv u_j$ on $(V_i\cap V_j)\times D$ for every $i,j\in \mathcal{J}$, $i\neq j$.
  \item[(e)] For every $i\in \mathcal{J}$  there exists an admissible nonsingular multidirection $\mathbf{d}$ such that
  $u_i= \Ss_{\mathbf{k},\mathbf{d}}\widehat{u}$ on $\tilde{V}\times D$ for some non empty sector $\tilde{V}\subseteq V_i$.
  \item[(f)] For every admissible nonsingular multidirection $\mathbf{d}$ there exists $i\in \mathcal{J}$ such that
  $\Ss_{\mathbf{k},\mathbf{d}}\widehat{u}= u_i$ on $\tilde{V}\times D$ for some sector $\tilde{V}\subseteq V_i$.
 \end{enumerate}
\end{Df}
 
Now we are ready to describe a maximal family of solutions of (\ref{eq:pde}) generalising our previous result \cite[Theorem 3]{Mic-Pod} to
the multisummable case.
\begin{Tw}
 \label{th:general}
 Let $\widehat{u}$ be a $\mathbf{k}$-multisummable normalised formal power series solution of (\ref{eq:pde})
 with a $\mathbf{k}$-multisum
 in a nonsingular admissible multidirection $\mathbf{d}$ given by
 $u^{\mathbf{d}}=\Ss_{\mathbf{k},\mathbf{d}}\widehat{u}$ and satisfying 
 $\widehat{u}=\widehat{u}_1+\cdots+\widehat{u}_n$, where $\widehat{u}_j$ is $k_j$-summable for $j=1,\dots,n$.
 Assume that there exists $q\in\QQ_+$, which is the smallest positive rational number such that
 $u^{\mathbf{d}}(t,z)=u^{\mathbf{d}}(te^{2q\pi i},z)$
 for every nonsingular multidirection $\mathbf{d}$.
 Suppose that the set of singular directions of $\widehat{u}$ of level $k_j$ modulo $2q\pi$ is given by $\{d_{j,1},\dots,d_{j,n_j}\}$,
 where $0\leq d_{j,1} < \dots< d_{j,n_j}<2q\pi$ ($j=1,\dots,n$).
 
 Furthermore, let
 \begin{gather*}
  I_{j,l}:=(d_{j,l}-\frac{\pi}{2k_j},d_{j, l+1}+\frac{\pi}{2k_j})\quad \textrm{for}\quad l=1,\dots,n_j,\quad j=1,\dots,n,
 \end{gather*}
where $d_{j, n_j+1}:=d_{j,1}+2q\pi$, and let
$$
 \mathcal{J}:=\{\mathbf{l}=(l_1,\dots,l_n)\in\NN^n\colon 1\leq l_j \leq n_j,\, |I_{j,l_j}\cap \dots\cap I_{n,l_n}|>\frac{\pi}{k_j},\, j=1,\dots,n\},
$$
where $|I|$ denotes the length of the interval $I$.

 Then:
 \begin{enumerate}
 \item[(i)] for every $\mathbf{l}\in \mathcal{J}$ there exists an admissible multidirection $\mathbf{d}=(d_1,\dots,d_n)$ satisfying
 \begin{gather}
 \label{eq:multi}
  d_j\in (d_{j,l_j},d_{j,l_{j}+1}),\quad j=1,\dots,n,
 \end{gather}
 for which the function
 $u_{\mathbf{l}}:=\Ss_{\mathbf{k},\mathbf{d}}\widehat{u}$ is well defined,
 \item[(ii)] for every sufficiently small $\varepsilon>0$ there exists $r>0$ such that $u_{\mathbf{l}}\in\OO(V_{\mathbf{l}}(\varepsilon,r)\times D)$
 for every $\mathbf{l}\in \mathcal{J}$,
 where
 $$V_{\mathbf{l}}(\varepsilon,r):=\{t\in W_r\colon (\arg t-\frac{\varepsilon}{2},\arg t +\frac{\varepsilon}{2})\subseteq I_{1,l_1}\cap \dots\cap 
 I_{n,l_n}\}$$
 and $W_r=\{t\in W\colon 0<|t|<r\}$ with $W$ being the Riemann surface of $t\mapsto t^{\frac{1}{q}}$,
 \item[(iii)] $\{u_{\mathbf{l}}\}_{\mathbf{l}\in \mathcal{J}}$ is a maximal family
 of solutions of (\ref{eq:pde}) on $W_r\times D$.
 \end{enumerate}
\end{Tw}

\begin{Uw}
 Observe that $\mathcal{L}_{d_{j,l}}$ and $\mathcal{L}_{d_{j,l}\pm\frac{\pi}{2k_j}}$, with $d_{j,l}$ satisfying the assumptions of
 Theorem \ref{th:general}, are respectively Stokes
 and anti-Stokes lines of level $j$ for $l=1,\dots,n_j$ and $j=1,\dots,n$. They play an important
 role in our description of the maximal family of solutions of (\ref{eq:pde}).
\end{Uw}

\begin{proof}[Proof of Theorem \ref{th:general}]
 (i) First, observe that condition
 \begin{gather*}
  |I_{j,l_j}\cap\dots\cap I_{n,l_n}|>\frac{\pi}{k_j}\quad\textrm{for}\quad j=1,\dots,n
 \end{gather*}
 guarantees that there exist $d_j\in(d_{j,l_j},d_{j,l_j+1})$, $j=1,\dots,n$, such that
 $$
 (d_1-\frac{\pi}{2k_1},d_1+\frac{\pi}{2k_1})\subseteq (d_2-\frac{\pi}{2k_2}, d_2+\frac{\pi}{2k_2})\subseteq\dots\subseteq
 (d_n-\frac{\pi}{2k_n}, d_n+\frac{\pi}{2k_n}).
 $$
 It means that for every $\mathbf{l}\in\mathcal{J}$ one can find an admissible multidirection 
 $\mathbf{d}=(d_1,\dots,d_n)$ satisfying (\ref{eq:multi}).
 
 Moreover, for every admissible multidirections $\mathbf{d}=(d_1,\dots,d_n)$ and
 $\mathbf{\tilde{d}}=(\tilde{d}_1,\dots,\tilde{d}_n)$ such that $d_j,\tilde{d}_j\in (d_{j,l_j},d_{j,l_{j}+1})$ for $j=1,\dots,n$ we have
 $$
 \Ss_{\mathbf{k},\mathbf{d}}\widehat{u}=\Ss_{k_1,d_1}\widehat{u}_1+\cdots+\Ss_{k_n,d_n}\widehat{u}_n=
 \Ss_{k_1,\tilde{d}_1}\widehat{u}_1+\cdots+\Ss_{k_n,\tilde{d}_n}\widehat{u}_n=\Ss_{\mathbf{k},\mathbf{\tilde{d}}}\widehat{u},
 $$
 since by \cite[Lemma 10]{B2} 
 \begin{gather*}
  \Ss_{k_j,d_j}\widehat{u}_j=\Ss_{k_j,\tilde{d}_j}\widehat{u}_j\quad \textrm{for every} \quad j=1,\dots,n.
 \end{gather*}
 It means that for every $\mathbf{l}\in\mathcal{J}$ the function $u_{\mathbf{l}}$ is well defined.
 
 To show (ii), observe that for every sufficiently small
 $\varepsilon>0$ there exists $r>0$ such that $\Ss_{k_j,d_j}\widehat{u}_j$ is analytically continued to the set 
 $$
 \{t\in W\colon |t|\in(0,r),\ (\arg t-\frac{\varepsilon}{2},\arg t +\frac{\varepsilon}{2})
 \subseteq I_{j,l_j}\}\times D\quad\textrm{for every}\quad j=1,\dots,n.
 $$
 Hence the whole function $u_{\mathbf{l}}$ is analytically continued to the set $V_{\mathbf{l}}(\varepsilon,r)\times D$.
 
 Finally we prove (iii). 
 Since the inequality $|I_{1,l_1}\cap \dots\cap I_{n,l_n}|>\frac{\pi}{k_1}$ holds for every $\mathbf{l}\in\mathcal{J}$,
 we are able to take such small $\varepsilon > 0$ that the opening of $V_{\mathbf{l}}(\varepsilon,r)$ ($V_{\mathbf{l}}$ for short)
 is greater than $\frac{\pi}{k_1}$.
 
 We claim that $\{V_{\mathbf{l}}\}_{\mathbf{l}\in\mathcal{J}}$ is a covering of $W_r$. To this end we take any $t\in W_r$. Then we may choose
 $\mathbf{l}\in\NN^n$ such that $1\leq l_j\leq n_j$ and $\arg t \in [d_{j,l_j},d_{j,l_j+1})$ for $j=1,\dots,n$.
 For such choice of $\mathbf{l}$ there exists $\delta>0$ such that
 $$
 [\arg t - \frac{\pi}{2k_j}+\frac{\delta}{2}, \arg t + \frac{\pi}{2k_j}+\delta] \subset I_{j,l_j}\cap\cdots\cap I_{n,l_n}\quad\textrm{for}\quad j=1,\dots,n.
 $$
 It means that $|I_{j,l_j}\cap \dots\cap I_{n,l_n}|\geq \frac{\pi}{k_j}+\frac{\delta}{2}>\frac{\pi}{k_j}$ for $j=1,\dots,n$, so $\mathbf{l}\in\mathcal{J}$
 and $t\in V_{\mathbf{l}}$. By the freedom
 of choice of $t\in W_r$, $\{V_{\mathbf{l}}\}_{\mathbf{l}\in\mathcal{J}}$ is a covering of $W_r$.

 By the moment version of \cite[Theorem 6.2]{B0} we conclude that the space of $\mathbf{k}$-multisummable series in a multidirection
 $\mathbf{d}$ is a moment differential algebra over $\CC$.
 It means that it is a linear space, which is also closed under multiplication
 and moment differentiations, and which for any $\mathbf{k}$-multisummable series $\widehat{f}$ and $\widehat{g}$ satisfies:
 $\Ss_{\mathbf{k},\mathbf{d}}(\widehat{f}+\widehat{g})=\Ss_{\mathbf{k},\mathbf{d}}\widehat{f}+\Ss_{\mathbf{k},\mathbf{d}}\widehat{g}$, $\Ss_{\mathbf{k},\mathbf{d}}(\widehat{f}\cdot\widehat{g})=\Ss_{\mathbf{k},\mathbf{d}}\widehat{f}\cdot\Ss_{\mathbf{k},\mathbf{d}}\widehat{g}$,
 $\Ss_{\mathbf{k},\mathbf{d}}(\partial_{m_1,t}\widehat{f})=
 \partial_{m_1,t}(\Ss_{\mathbf{k},\mathbf{d}}\widehat{f})$
 and $\Ss_{\mathbf{k},\mathbf{d}}(\partial_{m_2,z}\widehat{f})=
 \partial_{m_2,z}(\Ss_{\mathbf{k},\mathbf{d}}\widehat{f}$).
 
 Hence
  \begin{gather*}
  P(\partial_{m_1,t},\partial_{m_2,z})u_{\mathbf{l}}=P(\partial_{m_1,t},\partial_{m_2,z})\Ss_{\mathbf{k},\mathbf{d}}\widehat{u}=
  \Ss_{\mathbf{k},\mathbf{d}}P(\partial_{m_1,t},\partial_{m_2,z})\widehat{u}=0\ \textrm{on}\ V_{\mathbf{l}}\times D.
 \end{gather*}
 Additionally, since 
 $\widehat{u}(t,z)=\sum_{j=0}^{\infty}u_j(z)t^j$ on $V_{\mathbf{l}}\times D$, by \cite[Proposition 8]{B2} and by the definition
 of multisummable series we get
 \begin{gather*}
   \lim\limits_{t\to 0, t\in V_{\mathbf{l}}}\partial_{m_1,t}^ju_{\mathbf{l}}(t,z)=\frac{m_1(j)}{m_1(0)}u_j(z)=\varphi_j(z)\quad \textrm{for}\quad j=0,\dots,N-1.
 \end{gather*}
 Therefore $u_{\mathbf{l}}$ is an actual solution of (\ref{eq:pde}) for ${\mathbf{l}}\in\mathcal{J}$.
 
 Now, assume that $V_{\mathbf{l}}\cap V_{\tilde{\mathbf{l}}}\neq \emptyset$ and $u_{\mathbf{l}}\equiv u_{\tilde{\mathbf{l}}}$ on
 $(V_{\mathbf{l}}\cap V_{\tilde{\mathbf{l}}})\times D$ for some $\mathbf{l},\tilde{\mathbf{l}}\in\mathcal{J}$ and
 $\mathbf{l}\neq \tilde{\mathbf{l}}$. It means that there exists admissible multidirections $\mathbf{d}=(d_1,\dots,d_n)$, $d_j\in(d_{j,l_j},
 d_{j,l_{j+1}})$ and $\tilde{\mathbf{d}}=(\tilde{d}_1,\dots,\tilde{d}_n)$, $\tilde{d}_j\in(d_{j,\tilde{l}_j},d_{j,\tilde{l}_{j+1}})$ such that
 $$
 u_{\mathbf{l}}=\Ss_{k_1,d_1}\widehat{u}_1+\cdots+\Ss_{k_n,d_n}\widehat{u}_n=
 \Ss_{k_1,\tilde{d}_1}\widehat{u}_1+\cdots+\Ss_{k_n,\tilde{d}_n}\widehat{u}_n=u_{\tilde{\mathbf{l}}}.
 $$
 Since $\Ss_{k_j,d_j}\widehat{u}_j$ and $\Ss_{k_j,\tilde{d}_j}\widehat{u}_j$ are both analytic on the non-empty set 
 $$
 \{t\in W\colon |t|\in(0,r), (\arg t-\frac{\varepsilon}{2},\arg t +\frac{\varepsilon}{2})
 \subseteq I_{j,l_j}\cap I_{j,\tilde{l}_j}\}\times D,$$
 by the Relative Watson's lemma \cite[Proposition 2.1]{M-R} we conclude that
 $\Ss_{k_j,d_j}\widehat{u}_j=\Ss_{k_j,\tilde{d}_j}\widehat{u}_j$ for $j=1,\dots,n$.
 Since $\mathbf{l}\neq \tilde{\mathbf{l}}$, without loss of generality we may assume that $\tilde{d}_i<d_{i,l_i}<d_i$ for some $i\in\{1,\dots,n\}$.
This contradicts the fact that $d_{i,l_i}$ is a singular direction of level $k_i$.
So if
$V_{\mathbf{l}}\cap V_{\tilde{\mathbf{l}}}\neq \emptyset$ then $u_{\mathbf{l}}\not\equiv u_{\tilde{\mathbf{l}}}$ on
 $(V_{\mathbf{l}}\cap V_{\tilde{\mathbf{l}}})\times D$ for every $\mathbf{l},\tilde{\mathbf{l}}\in\mathcal{J}$,
 $\mathbf{l}\neq \tilde{\mathbf{l}}$.
 
 By the construction of the family $\{u_{\mathbf{l}}\}_{\mathbf{l}\in\mathcal{J}}$, the last two conditions in Definition \ref{df:maximal} are
 also satisfied, which completes the proof.
\end{proof}

\section{General linear moment partial differential equations with constant coefficients}
We will study the Stokes phenomenon and the maximal family of solutions for the normalised formal solution $\widehat{u}$ of (\ref{eq:pde}).
Let us recall that we may reduce the Cauchy problem (\ref{eq:pde}) of a general linear moment partial differential equation with constant
coefficients to
a family of the Cauchy problems of simple moment pseudodifferential equations. Namely we have
\begin{Stw}[{\cite[Theorem 1]{Mic8}}]
\label{pr:family}
Let $\widehat{u}$ be the normalised formal solution of (\ref{eq:pde}).
Then
$\widehat{u}=\sum_{\alpha=1}^l\sum_{\beta=1}^{N_{\alpha}}\widehat{u}_{\alpha\beta}$ with
   $\widehat{u}_{\alpha\beta}$ being a formal solution of a simple pseudodifferential equation
   \begin{equation*}
    \left\{
    \begin{array}{l}
     (\partial_{m_1,t}-\lambda_{\alpha}(\partial_{m_2,z}))^{\beta}u_{\alpha\beta}=0\\
     \partial_{m_1,t}^j u_{\alpha\beta}(0,z)=0\ \ (j=0,\dots,\beta-2)\\
     \partial_{m_1,t}^{\beta-1} \widehat{u}_{\alpha\beta}(0,z)=\lambda_{\alpha}^{\beta-1}(\partial_{m_2,z})\varphi_{\alpha\beta}(z),
    \end{array}
    \right.
   \end{equation*}
   where $\varphi_{\alpha\beta}(z):=\sum_{j=0}^{N-1}d_{\alpha\beta j}(\partial_{m_2,z})
   \varphi_j(z)\in\OO_{1/\gamma}(D)$ and $d_{\alpha\beta j}(\zeta)$ are some holomorphic
   functions of the variable $\xi=\zeta^{1/\gamma}$ and of polynomial growth. 
   \par
   Moreover, if $q_{\alpha}$ is a pole order of $\lambda_{\alpha}(\zeta)$
   and $\overline{q}_{\alpha}=\max\{0,q_{\alpha}\}$,
   then  $\widehat{u}_{\alpha\beta}\in\Oo_{1/\gamma}(D)[[t]]_{\overline{q}_{\alpha}s_2 - s_1}$.
\end{Stw}
 
For this reason we will study the following simple moment pseudodifferential equation 
\begin{equation}
 \label{eq:simple}
 \begin{cases}
  (\partial_{m_1,t}-\lambda(\partial_{m_2,z}))^{\beta}u=0\\
  \partial_{m_1,t}^j u(0,z)=0\ (j=0,\dots,\beta-2)\\
  \partial_{m_1,t}^{\beta-1}u(0,z)=\lambda^{\beta-1}(\partial_{m_2,z})\varphi(z)\in\OO_{1/\gamma}(D),
 \end{cases}
\end{equation}
where $m_1,m_2$ are moment functions of orders respectively $s_1,s_2>0$ such that $qs_2>s_1$, $\gamma\in\NN$,
$\lambda(\zeta)\sim \lambda_0\zeta^q$ with
$q=\mu/\nu$ for some relatively prime $\mu,\nu\in\NN$ satisfying $q\gamma\in\NN$.

We start from the following representation of summable solutions of (\ref{eq:simple}).

\begin{Stw}
\label{pr:sum}
 Let $d\in\RR$, $K=(qs_2-s_1)^{-1}$, $\varepsilon>0$ and $m(u)$ be a moment function of order $1/K$. Suppose that
 $\widehat{u}(t,z)\in\Oo_{1/\gamma}(D)[[t]]$
 is the unique formal power series solution of the Cauchy problem (\ref{eq:simple}) and
 \begin{equation}
 \label{eq:condition}
 \varphi(z)\in\OO_{1/\gamma}^{qK}\big(\bigcup_{l=0}^{q\gamma-1}\widehat{S}_{(d+\arg\lambda_0+2l\pi)/q}(\varepsilon/q)\big).
 \end{equation}
 Then $\widehat{u}(t,z)$ is $K$-summable in the direction $d$ and for every $\tilde{d}\in (d-\frac{\varepsilon}{2}, d+\frac{\varepsilon}{2})$
 and for every $\tilde{\varepsilon}\in(0,\varepsilon)$ there exists $r>0$ such that its $K$-sum
 $u^{\tilde{d}}\in\OO_{1,1/\gamma}(S_{\tilde{d}}(\pi/K-\tilde{\varepsilon},r)\times D)$ is given by
 \begin{equation}
  \label{eq:actual}
  u^{\tilde{d}}(t,z)=
  \mathcal{S}_{K,\tilde{d}}\widehat{u}(t,z)=(T_{m,\tilde{d}}v)(t,z)
  =\int_{e^{i\tilde{d}}\RR_+}e_{m}(s/t)v(s,z)\frac{ds}{s}, 
 \end{equation}
 where $v(t,z)=\Bo_{m}\widehat{u}(t,z)$
 has the integral representation
 \begin{equation}
  \label{eq:v_integral}
  v(t,z)=\frac{t^{\beta-1}}{(\beta-1)!}\partial_t^{\beta-1}
   \frac{1}{2\gamma\pi i}\oint_{|w|=\varepsilon}^{\gamma}\varphi(w)\int_{r_0e^{i\theta}}^{\infty e^{i\theta}}
   E_{\overline{m}_1}(t\lambda(\zeta))
    E_{\tilde{m}_2}(\zeta^{1/\gamma} z^{1/\gamma})\frac{e_{m_2}(\zeta w)}{\zeta w}\,d\zeta\,dw
 \end{equation}
with $\overline{m}_1(u)=m_1(u)m(u)$, $\tilde{m}_2(u)=m_2(u/\gamma)$ and $\theta\in(-\arg w - \frac{s_2\pi}{2},-\arg w + \frac{s_2\pi}{2})$.
\end{Stw}
\begin{proof}
 First, observe that by Proposition \ref{pr:family} we get $\widehat{u}(t,z)\in \OO_{1/\gamma}(D)[[t]]_{qs_2-s_1}$.
 Hence the function $v(t,z):=\Bo_m\widehat{u}(t,z)$ belongs to the space $\OO_{1,1/\gamma}(D^2)$.
 Moreover, by Proposition \ref{pr:commutation} it 
 satisfies
 \begin{equation*}
  \begin{cases}
   (\partial_{\overline{m}_1,t}-\lambda(\partial_{m_2,z}))^{\beta}v=0\\
   \partial_{\overline{m}_1,t}^j v(0,z)=0\ (j=0,\dots,\beta-2)\\
  \partial_{\overline{m}_1,t}^{\beta-1}v(0,z)=\lambda^{\beta-1}(\partial_{m_2,z})\varphi(z)\in\OO_{1/\gamma}(D).
  \end{cases}
 \end{equation*}
Hence by \cite[Lemma 3]{Mic8} we get the integral representation (\ref{eq:v_integral})
of $v(t,z)$.

Since $\varphi(z)$ satisfies (\ref{eq:condition}), by \cite[Lemma 4]{Mic8} we conclude that 
$v(t,z)\in\OO_{1,1/\gamma}^{K}(\widehat{S}_d(\varepsilon)\times D)$.
So, for every $\tilde{d}\in (d-\frac{\varepsilon}{2}, d+\frac{\varepsilon}{2})$ the function
$u^{\tilde{d}}(t,z):=T_{m,\tilde{d}}v(t,z)$ is well-defined and by the definitions of kernel functions
(Definitions \ref{df:moment} and \ref{df:small})
for every $\tilde{\varepsilon}\in(0,\varepsilon)$ there exists $r>0$ such that 
 $u^{\tilde{d}}\in\OO_{1,1/\gamma}(S_{\tilde{d}}(\pi/K-\tilde{\varepsilon},r)\times D)$. 
\end{proof}

Now we are ready to describe the Stokes phenomenon and the maximal family of solutions of the simple moment pseudodifferential
equation (\ref{eq:simple}) with the Cauchy data having the separate singular point at $z_0\in\CC\setminus\{0\}$.

\begin{Tw}
 \label{th:summability}
 Let $\widehat{u}$ be a formal solution of (\ref{eq:simple}) with $\varphi\in\OO_{1/\gamma}^{qK}(\widetilde{\CC\setminus\{z_0\}})$
 for some $z_0\in\CC\setminus\{0\}$.
 Set $K:=(qs_2-s_1)^{-1}$, $\delta_l:=q\arg z_0+\frac{2l\pi}{\nu}-\arg\lambda_0$ and $u_l:=u^{\tilde{d}}$ for 
 $\tilde{d}\in (\delta_l,\delta_{l+1})\mod 2q\pi$ (for $l=0,\dots,\mu-1$ with $\delta_{\mu}:=\delta_{0}+2q\pi$),
 where $u^{\tilde{d}}$ is  given by (\ref{eq:actual}). Finally, let $W_r=\{t\in W\colon 0<|t|<r\}$ for $r>0$, where $W$ is the Riemann surface
 of the function $t\mapsto t^{\frac{1}{q}}$.
 
 Then for every $\tilde{\varepsilon}>0$ there exists $r>0$ such that
 $u_l\in\OO_{1,1/\gamma}(S_{\delta_l+\frac{\pi}{\nu}}( (K^{-1}+\frac{2}{\nu})\pi - \tilde{\varepsilon}, r)\times D)$
 ($l=0,\dots,\mu-1$)
 and $\{u_0,\dots,u_{\mu-1}\}$ is a maximal family of solutions of (\ref{eq:simple}) on $W_r\times D$.
 
 Moreover, the sets $\mathcal{L}_{\delta_l}$ and $\mathcal{L}_{\delta_l\pm\frac{\pi}{2K}}$ ($l=0,\dots,\mu-1$) are respectively
 Stokes lines
 and anti-Stokes lines for $\widehat{u}$. The jump across the Stokes line $\mathcal{L}_{\delta_l}$ is given by
 \begin{equation*}
  J_{\mathcal{L}_{\delta_l}}\widehat{u}(t,0)= u_l(t,0)-u_{l-1}(t,0)=u^{\delta_l^+}(t,0)-u^{\delta_l^-}(t,0)
  =F_l(s,0)\Big[\frac{e_m(s/t)}{s}\Big],
 \end{equation*}
where $F_l(s,0)\in \mathcal{H}^{qK}(\mathcal{L}_{\delta_l})$ is defined by $F_l(s,0):=[v(s,0)]_{\delta_l}$ and
$v(s,z)=\Bo_m\widehat{u}(s,z)$ has the representation (\ref{eq:v_integral}).
\end{Tw}
\begin{proof}
First observe that, if $d\neq\delta_l\mod 2q\pi$ for $l=0,\dots,\mu-1$ then $\varphi$ satisfies the assumption (\ref{eq:condition})
for sufficiently small $\varepsilon>0$.
Hence by Proposition \ref{pr:sum}, $\widehat{u}$ is $K$-summable in a direction $\tilde{d}\in\RR$,
$\tilde{d}\neq \delta_l\mod 2q\pi$  for $l=0,\dots,\mu-1$
and its $K$-sum $u^{\tilde{d}}(t,z)$ satisfies (\ref{eq:actual}).

Observe that $u^{\tilde{d}}(t,z)=u^{\tilde{d}}(te^{2q\pi i},z)$ and $q$ is the smallest positive rational number for which
this equality holds.
Moreover, the set of singular directions of $\widehat{u}(t,z)$ modulo $2q\pi$ is given by $\{\delta_l\mod 2q\pi\colon l=0,\dots,\mu-1\}$.
Hence by \cite[Theorem 3]{Mic-Pod}, for every $\tilde{\varepsilon}>0$ there exists $r>0$ such that
$\{u_0,\dots,u_{\mu-1}\}$ with $u_l\in\OO_{1,1/\gamma}(S_{\delta_l+\frac{\pi}{\nu}}( (K^{-1}+\frac{2}{\nu})\pi - \tilde{\varepsilon}, r)\times D)$
($l=0,\dots,\mu-1$) is a maximal family
of solutions of (\ref{eq:simple}).
Moreover,
Stokes lines for $\widehat{u}$ are the sets
$\mathcal{L}_{\delta_l}$ and anti-Stokes lines for $\widehat{u}$ are the sets
$\mathcal{L}_{\delta_l\pm\frac{\pi}{2K}}$.

Now we are ready to calculate the jump across Stokes line $\mathcal{L}_{\delta_{l}}$.
\begin{equation*}
 J_{\mathcal{L}_{\delta_l}}\widehat{u}(t,z)=u^{\delta_l^+}(t,z)-u^{\delta_l^-}(t,z)
 \stackrel{(\ref{eq:actual})}{=}\int_{e^{i\delta_l^+}\RR_+}e_{m}(s/t)v(s,z)\frac{ds}{s}
 -\int_{e^{i\delta_l^-}\RR_+}e_{m}(s/t)v(s,z)\frac{ds}{s}.
\end{equation*}
Hence
$$ J_{\mathcal{L}_{\delta_l}}\widehat{u}(t,0)=F_l(s,0)\Big[\frac{e_{m}(s/t)}{s}\Big],
$$
where $F_l(s,0)\in \mathcal{H}^{qK}(\mathcal{L}_{\delta_l})$ is a hyperfunction on $\mathcal{L}_{\delta_l}$ defined by
$F_l(s,0):=[v(s,0)]_{\delta_l}$.
\end{proof}

Now we return to the general equation (\ref{eq:pde}).
For convenience we assume that
 \begin{equation}
   \label{eq:P_multi}
   P(\lambda,\zeta)=P_0(\zeta)\tilde{P}(\lambda,\zeta)=P_0(\zeta)\prod_{i=1}^{\tilde{n}}\prod_{\alpha=1}^{l_i}
   (\lambda-\lambda_{i\alpha}(\zeta))^{N_{i\alpha}},
 \end{equation}
  where $\lambda_{i\alpha}(\zeta)\sim \lambda_{i\alpha}\zeta^{q_i}$ is the root of the characteristic equation
  with $q_i\in\QQ$
  and $\lambda_{i\alpha}\in\CC\setminus\{0\}$ for
  $i=1,\dots,\tilde{n}$ and $\alpha=1,\dots,l_i$.
 Without loss of generality we may assume that there exist exactly $n$, $n\leq \tilde{n}$, pole orders $q_i$, which are greater than $s_1/s_2$,
 where $s_1,s_2$ are orders of moment functions $m_1,m_2$ respectively. We also assume that $s_1/s_2<q_1<\cdots<q_n<\infty$ and let
$K_i:=(q_is_2-s_1)^{-1}$ for $i=1,\dots,n$. Under the above conditions we have

\begin{Tw}
 \label{th:multisummability}
 Let $\widehat{u}$ be a normalised formal solution of (\ref{eq:pde}), $z_0\in\CC\setminus\{0\}$ and
 $\varphi_j(z)\in\OO^{q_nK_n}(\widetilde{\CC\setminus\{z_0\}})$ for $j=0,\dots,N-1$.
 Let $Q:=\frac{LCM(\mu_1,\dots,\mu_n)}{GCD(\nu_1,\dots,\nu_n)}$ and let
 $$
 \Lambda_i:=\{\delta\colon \delta=q_i\arg z_0 + \frac{2j\pi}{\nu_i}-\arg\lambda_{i\alpha} \mod 2Q\pi,\ 0\leq j\leq Q\nu_i-1,\ 1\leq\alpha\leq l_i\}
 $$
 for $i=1,\dots,n$. It means that we may assume that there exist $0\leq \delta_{i,1}<\dots < \delta_{i,n_i}<2Q\pi$ such that
 $\Lambda_i=\{\delta_{i,1},\dots,\delta_{i,n_i}\}$. Moreover let
 $$
 I_{i,j}:=(\delta_{i,j}-\frac{\pi}{2K_i},\delta_{i,j+1}+\frac{\pi}{2K_i})\quad\textrm{for}\quad j=1,\dots,n_i,\quad i=1,\dots,n,
 $$
 with $\delta_{i,n_i+1}:=\delta_{i,1}+2Q\pi$, 
 \begin{equation*}
 \mathcal{J}:=\{\mathbf{l}=(l_1,\dots,l_n)\in\NN^n\colon 1\leq l_i \leq n_i,\
 |I_{i,l_i}\cap \dots\cap I_{n,l_n}|>\frac{\pi}{K_i}\ \textrm{for}\ i=1,\dots,n\},
 \end{equation*}
 $\mathbf{K}:=(K_1,\dots,K_n)$
 and $W_r=\{t\in W\colon 0<|t|<r\}$ for $r>0$, where $W$ is the Riemann surface of the function $t\mapsto t^{1/Q}$.
 
 Then the following conditions holds:
 \begin{enumerate}
 \item[(a)] The formal solution $\widehat{u}$ is $\mathbf{K}$-multisummable, $\widehat{u}=\widehat{u}_0+\widehat{u}_1+\dots+\widehat{u}_n$,
 where $\widehat{u}_0$ is a convergent power series solution of
 \begin{equation}
 \label{eq:u_0}
 \big(\prod_{i=n+1}^{\tilde{n}}\prod_{\alpha=1}^{l_i}(\partial_{m_1,t}-\lambda_{i\alpha}(\partial_{m_2,z}))^{N_{i\alpha}}\big)u_0=0
 \end{equation}
 and $\widehat{u}_i$ is a $K_i$-summable power series solution of 
 \begin{equation}
 \label{eq:u_i}
 \big(\prod_{\alpha=1}^{l_i}(\partial_{m_1,t}-\lambda_{i\alpha}(\partial_{m_2,z}))^{N_{i\alpha}}\big)u_i=0\quad\textrm{for}\quad i=1,\dots,n.
 \end{equation}
 Moreover 
 \begin{equation}
 \label{eq:multisum}
 \mathcal{S}_{\mathbf{K},\mathbf{d}}\widehat{u}=u_0+\mathcal{S}_{K_1,d_1}\widehat{u}_1+\dots+\mathcal{S}_{K_n,d_n}\widehat{u}_n
 \end{equation}
 for any admissible nonsingular multidirection $\mathbf{d}=(d_1,\dots,d_n)$.
 \item[(b)] For every $\mathbf{l}\in\mathcal{J}$ the function $u_{\mathbf{l}}(t,z):=\mathcal{S}_{\mathbf{K},\mathbf{d}}\widehat{u}$ is
 a well defined actual solution of (\ref{eq:pde}),
 where $\mathbf{d}$ is an admissible nonsingular multidirection satisfying $d_i\in(\delta_{i,l_i},\delta_{i,l_i+1})$
 for $i=1,\dots,n$,.
 \item[(c)] For every $\varepsilon>0$ there exists $r>0$ such that $u_{\mathbf{l}}\in\OO_{1,1/\gamma}(V_{\mathbf{l}}(\varepsilon,r)\times D)$, where
 $$
 V_{\mathbf{l}}(\varepsilon,r):=\{t\in W_r\colon (\arg t - \frac{\varepsilon}{2},\arg t + \frac{\varepsilon}{2})\subseteq 
 I_{1,l_1}\cap\dots\cap I_{n,l_n}\}.
 $$
 \item[(d)] $\{u_{\mathbf{l}}\}_{\mathbf{l}\in\mathcal{J}}$ is a maximal family of solutions of (\ref{eq:pde}).
 \item[(e)] For every $i\in\{1,\dots,n\}$ the sets $\mathcal{L}_{\delta_{i,j}}$
  (resp.  $\mathcal{L}_{\delta_{i,j}\pm\frac{\pi}{2K_i}}$), $j=1,\dots,n_i$ are Stokes lines (resp. anti-Stokes lines) of level
  $K_i$.
 \item[(f)] For every $i\in\{1,\dots,n\}$ and $j\in\{1,\dots,n_i\}$, the jump across the Stokes line $\mathcal{L}_{\delta_{i,j}}$
 of level $K_i$  is
 given by
 \begin{equation*}
  J_{\mathcal{L}_{\delta_{i,j}},K_i}\widehat{u}(t,0)=u_{\mathbf{l}}(t,0)-u_{\mathbf{l}'}(t,0)=u_i^{\delta_{i,j}^+}(t,0)-u_i^{\delta_{i,j}^-}(t,0)
  =F_{i,j}(s,0)\Big[\frac{e_{\overline{m}_i}(s/t)}{s}\Big],
 \end{equation*}
 where $u_i^{\delta_{i,j}^{\pm}}= \mathcal{S}_{K_i,\delta_{i,j}^{\pm}}\widehat{u}_i$, $\overline{m}_i$ is a moment function of order $1/K_i$,
 $F_{i,j}(s,0)$ is a hyperfunction
 on $\mathcal{L}_{\delta_{i,j}}$ defined by
 $F_{i,j}(s,0):=[v_i(s,0)]_{\delta_{i,j}}$,
 $v_i(s,z):=\Bo_{\overline{m}_i}\widehat{u}_i(s,z)$ and $\mathbf{l},\mathbf{l}'\in \mathcal{J}$ satisfy
 $l_i'=j-1$ in the case when $l_i=j$ and $j>1$,
 $l_i'=n_i$ in the case when $l_i=1$,
 and $l_{\alpha}=l_{\alpha}'$ for $\alpha\neq i$.
\end{enumerate}
Moreover, under the additional condition that $\widehat{u}$ is a normalised formal solution of (\ref{eq:pde}) with $m_2(u)=\Gamma(1+u)$
(i.e. when $\widehat{u}$ is a normalised formal solution of the Cauchy problem
$P(\partial_{m_1,t},\partial_z)u=0$, $\partial_{m_1,t}^ju(0,z)=\varphi_j(z)$ for
$j=0,\dots,N-1$), we may replace the assertion (f) by
\begin{enumerate}
 \item[(f')] For every $i\in\{1,\dots,n\}$, $j\in\{1,\dots,n_i\}$ and $z\in D$, the jump across the Stokes line $\mathcal{L}_{\delta_{i,j}}$
 of level $K_i$  is given by
 \begin{equation*}
  J_{\mathcal{L}_{\delta_{i,j}},K_i}\widehat{u}(t,z)=u_{\mathbf{l}}(t,z)-u_{\mathbf{l}'}(t,z)=u_i^{\delta_{i,j}^+}(t,z)-u_i^{\delta_{i,j}^-}(t,z)
  =F_{i,j}(s,z)\Big[\frac{e_{\overline{m}_i}(s/t)}{s}\Big],
 \end{equation*}
 where $u_i^{\delta_{i,j}^{\pm}}= \mathcal{S}_{K_i,\delta_{i,j}^{\pm}}\widehat{u}_i$, $\overline{m}_i$ is a moment function of order $1/K_i$, $F_{i,j}(s,z)$ is a hyperfunction
 on $\mathcal{L}_{\delta_{i,j}(z)}$ defined by
 $F_{i,j}(s,z):=[v_i(s,z)]_{\delta_{i,j}(z)}$, $\delta_{i,j}(z):=\delta_{i,j}+q_i(\arg(z_0-z)-\arg z_0)$,
 $v_i(s,z):=\Bo_{\overline{m}_i}\widehat{u}_i(s,z)$ and $\mathbf{l},\mathbf{l}'\in \mathcal{J}$ satisfy the same conditions as in (f).
\end{enumerate}
 \end{Tw}
 \begin{proof}
  Since $P(\lambda,\zeta)$ is given by (\ref{eq:P_multi}), by Proposition \ref{pr:family} a normalised formal solution $\widehat{u}$ of
  (\ref{eq:pde}) may be
written as $\widehat{u}=\widehat{u}_0+\widehat{u}_1+\dots+\widehat{u}_n$, where $\widehat{u}_0$ is a convergent power series solution of the
pseudodifferential equation (\ref{eq:u_0})
and $\widehat{u}_i$ is a $1/K_i$-Gevrey power series solution of (\ref{eq:u_i})
with the initial data having the same holomorphic properties as $\varphi_j(z)$.
More precisely, by Proposition \ref{pr:family} we conclude that
$\widehat{u}_i=\sum_{\alpha=1}^{l_i}\sum_{\beta=1}^{N_{i\alpha}}\widehat{u}_{i\alpha\beta}$,
where $\widehat{u}_{i\alpha\beta}$ is a formal solution of a simple pseudodifferential equation
 \begin{equation*}
    \left\{
    \begin{array}{l}
     (\partial_{m_1,t}-\lambda_{i\alpha}(\partial_{m_2,z}))^{\beta}u_{i\alpha\beta}=0\\
     \partial_{m_1,t}^j u_{i\alpha\beta}(0,z)=0\ \ (j=0,\dots,\beta-2)\\
     \partial_{m_1,t}^{\beta-1} \widehat{u}_{i\alpha\beta}(0,z)=\lambda_{i\alpha}^{\beta-1}(\partial_{m_2,z})\varphi_{i\alpha\beta}(z),
    \end{array}
    \right.
   \end{equation*}
   where $\varphi_{i\alpha\beta}(z):=\sum_{j=0}^{N-1}d_{i\alpha\beta j}(\partial_{m_2,z})
   \varphi_j(z)\in\OO_{1/\gamma}(D)$ and $d_{i\alpha\beta j}(\zeta)$ are some holomorphic
   functions of the variable $\xi=\zeta^{1/\gamma}$ and of polynomial growth. Since $\varphi_j(z)\in\OO^{q_nK_n}(\widetilde{\CC\setminus\{z_0\}})$
and $q_iK_i\leq q_nK_n$ we see that
$\varphi_{i\alpha\beta}(z)\in\Oo_{1/\gamma}^{q_iK_i}\big(\bigcup_{l=0}^{q_i\gamma-1}\widehat{S}_{(d+\arg\lambda_{\alpha}+2l\pi)/q_i}\big)$
for every direction $d\neq q_i\arg z_0 + \frac{2j\pi}{\nu_i}-\arg\lambda_{i\alpha}\mod 2\pi q_i$, $j=0,\dots,\mu-1$. 

Hence, by Proposition
\ref{pr:sum} and Theorem \ref{th:summability} we see that $\widehat{u}_{i\alpha\beta}$ is $K_i$-summable with the singular directions given by
$q_i\arg z_0 + \frac{2j\pi}{\nu_i}-\arg\lambda_{i\alpha}\mod 2\pi q_i$ for $j=0,\dots \mu_i-1$. Consequently,
$\widehat{u}$ is
$\mathbf{K}$-multisummable in any nonsingular admissible multidirection $\mathbf{d}=(d_1,\dots,d_n)$. Since a formal power series  $\widehat{u}_0$ 
is convergent, its sum $u_0$ is well defined and by Remark \ref{re:suma} we conclude that
$\mathbf{K}$-multisum $\mathcal{S}_{\mathbf{K},\mathbf{d}}\widehat{u}$ of $\widehat{u}$ is given by (\ref{eq:multisum}), so (a) holds.

Since $\widehat{u}$ is $K$-multisummable,
using Theorem \ref{th:general} we conclude that (b), (c) and (d) hold.

Since the set of singular directions of order $K_i$ is given by $\Lambda_i$, we get the description of Stokes lines $\mathcal{L}_{\delta_{i,j}}$
and anti-Stokes lines $\mathcal{L}_{\delta_{i,j}\pm\frac{\pi}{2K_i}}$ of level $K_i$ for $\delta_{i,j}\in\Lambda_i$ and $i=1,\dots,n$,
so (e) is also satisfied.

Finally, to obtain (f) by Theorem \ref{th:summability} we calculate the jumps for $\widehat{u}$ across
the Stokes lines $\mathcal{L}_{d_i}$ of level $K_i$. Using Remark \ref{re:z}, we get (f').
 \end{proof}

 Let us illustrate our theory on the following simple example.
 
 \begin{Prz}
  Let $\widehat{u}$ be a formal solution of the Cauchy problem
 \begin{equation}
 \label{eq:ex}
  \begin{cases}
   (\partial_t-\partial_z^2)(\partial_t-\partial_z^3)\widehat{u}=0,\\
   u(0,z)=\varphi_1(z)\in\OO^{3/2}(\widetilde{\CC\setminus\{z_0\}}),\\
   \partial_t u(0,z)=\varphi_2(z)\in\OO^{3/2}(\widetilde{\CC\setminus\{z_0\}}),
  \end{cases}
\end{equation}
where $\arg z_0\in(-\frac{\pi}{2},\frac{\pi}{2})$.

 Then by Proposition \ref{pr:family} (see also \cite[Theorem 1]{Mic8}) $\widehat{u}=\widehat{u}_1+\widehat{u}_2$,
 where $\widehat{u}_1$, $\widehat{u}_2$ are formal solutions of
 $$
 (\partial_t-\partial_{z}^2)u_1=0,\quad u_1(0,z)=c_{11}(\partial_z)\varphi_1(z)+c_{12}(\partial_z)\varphi_2(z)=:\tilde{\varphi}_1(z),
 $$
 $$
 (\partial_t-\partial_{z}^3)u_2=0,\quad u_2(0,z)=c_{21}(\partial_z)\varphi_1(z)+c_{22}(\partial_z)\varphi_2(z)=:\tilde{\varphi}_2(z),
 $$
 where $c_{ij}(\partial_z)$, $i,j=1,2$, are pseudodifferential operators defined by $c_{11}(\zeta):=\frac{\zeta}{\zeta-1}$,
 $c_{12}(\zeta):=\frac{1}{\zeta^2-\zeta^3}$, $c_{21}(\zeta):=\frac{1}{1-\zeta}$ and $c_{22}(\zeta):=\frac{1}{\zeta^3-\zeta^2}$.
 \par
 Now we are ready to describe the Stokes phenomenon and the maximal family of solutions of (\ref{eq:ex}). 
 By Theorem \ref{th:summability}, $\widehat{u}_1$ is $1$-summable with singular directions $d_{1,l}:=2\arg z_0+2l\pi$
 and $\widehat{u}_2$ is $1/2$-summable with singular directions $d_{2,l}:=3\arg z_0+2l\pi$ for $l\in\ZZ$.
 Hence $\widehat{u}$ is $(1,1/2)$-summable, the set of singular directions (modulo $12\pi$) of level $1$ is given by $\{d_{1,0},\dots,d_{1,5})$
 and the set of singular directions (modulo $12\pi$) of level $1/2$ is given by $\{d_{2,0},\dots,d_{2,5})$.
 It means that $\mathcal{L}_{d_{1,l}}$ and $\mathcal{L}_{d_{1,l}\pm\frac{\pi}{2}}$ are respectively Stokes and anti-Stokes lines of level $1$,
 and analogously $\mathcal{L}_{d_{2,l}}$ and $\mathcal{L}_{d_{2,l}\pm\pi}$ are respectively Stokes and anti-Stokes lines of level $1/2$
 ($l=0,\dots,5$).

  Next, let $I_{1,j}:=(-\frac{\pi}{2}+2\pi j +2\arg z_0, \frac{\pi}{2}+2\pi(j+1)+2\arg z_0)$ for $j=0,\dots,5$ and
  $I_{2,k}:=(-\pi+2\pi k +3\arg z_0, \pi+2\pi(k+1)+3\arg z_0)$ for $k=0,\dots,5$.
  Since $\arg z_0\in(-\frac{\pi}{2},\frac{\pi}{2})$, we conclude that 
  $$|I_{1,j}\cap I_{2,k}|=
  \begin{cases}
   3\pi &|j-k|=0\\
   \frac{3}{2}\pi\pm\arg z_0 & j-k=\pm 1\\
   0 & |j-k|>1.
  \end{cases}
  $$
  
 Hence $\mathcal{J}=\{(j,k)\colon\ 0\leq j,k\leq 5,\ |j-k|\leq 1\}$ and $\{u_{(j,k)}\}_{(j,k)\in\mathcal{J}}$ is a maximal family of solutions of
 (\ref{eq:ex}) on the Riemann surface of $t\mapsto t^{\frac{1}{6}}$, where $u_{(j,k)}:=u_{1,j}+u_{2,k}$,
 $u_{1,j}:=u_1^d=\mathcal{S}_{1,d}\widehat{u}_1$ for $d\in (d_{1,j},d_{1,j+1})=(2\arg z_0+2\pi j, 2\arg z_0+2\pi(j+1))$ and
 $u_{2,k}:=u_2^d=\mathcal{S}_{2,k}\widehat{u}_2$ for $d\in (d_{2,j},d_{2,j+1})=(3\arg z_0+2\pi k, 3\arg z_0+2\pi(k+1))$.
 
 Using \cite{Mic-Pod} and \cite{Tk} we are also able to calculate the jumps across the Stokes lines. Namely
 \begin{multline*}
 J_{\mathcal{L}_{d_{1,j},1}}\widehat{u}(t,z)= u_{(j,j)}(t,z)-u_{(j-1,j)}(t,z)=u_{1,j}(t,z)-u_{1,j-1}(t,z)\\
 =J_{\mathcal{L}_{d_{1,j}}}\widehat{u}_1=
 u_1^{d_{1,j}^+}(t,z)-u_1^{d_{1,j}^-}(t,z)=F_{1,z}(s)[\frac{1}{\sqrt{4\pi t}}e^{-\frac{s^2}{4t}}],
 \end{multline*}
 where $F_{1,z}(s):=[\tilde{\varphi}_1(z+s)]_{\arg(z_0-z)}$ is a hyperfunction on $\{s\in\tilde{\CC}\colon \arg s = \arg(z_0-z)\}$.
 
 Analogously
 \begin{multline*}
 J_{\mathcal{L}_{d_{2,j},1/2}}\widehat{u}(t,z)= u_{(j,j)}(t,z)-u_{(j,j-1)}(t,z)=u_{2,j}(t,z)-u_{2,j-1}(t,z)\\
 =J_{\mathcal{L}_{d_{1,j}}}\widehat{u}_2=
 u_2^{d_{2,j}^+}(t,z)-u_2^{d_{2,j}^-}(t,z)=F_{2,z}(s)[\frac{1}{3\sqrt[3]{t}}C_3(s/\sqrt[3]{t})],
 \end{multline*}
 where $F_{2,z}(s):=[\tilde{\varphi}_2(z+s)]_{\arg(z_0-z)}$ is a hyperfunction on $\{s\in\tilde{\CC}\colon \arg s = \arg(z_0-z)\}$
 and $C_3(\tau)$ is the Ecalle kernel defined by $C_3(\tau):=\sum_{n=0}^{\infty}\frac{(-\tau)^n}{n!\Gamma(1-\frac{n+1}{3})}$.
 \end{Prz}

\section{Moment partial differential equations --- special cases}
In this section we will consider certain special cases of moment partial differential equations. We derive Stokes lines and jumps across these Stokes
lines in terms of hyperfunctions.
\bigskip

\emph{Case 1.} Let us consider the following equation
$$\left\{ \begin{array}{ll} \partial_{m_1,t}^{p}u(t,z)=\partial_{z}^{q}u(t,z)\,\,\,\mathrm{with}\ 0< ps_1<q,\\ 
u(0,z)=\varphi(z),\,\, \\ 
\partial_{m_1,t}^{j}u(0,z)=0,\,\, \textrm{for}\,\,\, j=1,2,\dots,p-1
\end{array}\right.$$
with $\varphi(z)\in\OO^{\frac{q}{q-ps_1}}\Bigl(\widetilde{\CC\setminus\{z_0\}}\Bigr)$ for some $z_0\in\CC\setminus\{0\}$;
where $m_1$ is a moment function of order $s_1>0$ corresponding to a kernel function $e_{m_1}(z)$ of order $1/s_1$.

The above Cauchy problem has a unique formal solution $$\widehat{u}(t,z)=\sum_{n=0}^{\infty}\frac{\varphi^{(qn)}(z)}{m_1(pn)}\* t^{pn},$$ 
to which we first apply the $m$-moment Borel transform. We obtain
\begin{multline*}
(\Bo_m\widehat{u})(t,z)=\sum_{n=0}^{\infty}\frac{\varphi^{(qn)}(z)}{m_1(pn)}\cdot\frac{m_1(pn)}{\Gamma(1+\frac{q}{p}\cdot pn)}
\*t^{(\frac{p}{q})\*qn}=\sum_{n=0}^{\infty}\frac{\varphi^{(qn)}(z)}{(qn)!}\*t^{(\frac{p}{q})\*qn}\\
=\frac{1}{q}\*\bigl(\varphi(z+\sqrt[q]{t^p})+\varphi(z+e^{\frac{2\pi\*i}{q}}\*\sqrt[q]{t^{p}})+\dots+\varphi(z+e^{\frac{2(q-1)\pi\*i}{q}}
\*\sqrt[q]{t^p})\bigr),
\end{multline*}
where $m(n):=\frac{\Gamma(1+\frac{q}{p}n)}{m_1(n)}=\frac{\Gamma_{\frac{q}{p}}(n)}{m_1(n)}$ is a moment function of order $\frac{q}{p}-s_1$ corresponding to a kernel function $e_m(z)$ of order $\frac{p}{q-ps_1}.$

Let $f(s,z):=(\Bo_m\widehat{u})(s,z)$, then by using $m$-moment Laplace transform in a nonsingular direction $d$ we get 
\begin{multline*} (T_{m,d}f)(t,z)=\int_{e^{id}\RR_+}e_m(s/t)\*f(s,z)\frac{ds}{s}\\
 =\frac{1}{q}\int_{e^{id}\RR_+}e_m(s/t)\*\*\biggl(\varphi(z+\sqrt[q]{s^p})+\varphi(z+e^{\frac{2\pi\*i}{q}}\*\sqrt[q]{s^{p}})+\dots+\varphi(z+e^{\frac{2(q-1)\pi\*i}{q}}\*\sqrt[q]{s^p})\biggr)\frac{ds}{s}.
 \end{multline*}

Thus, by Proposition \ref{pr:sum}, the unique formal solution $\widehat{u}(t,z)$ of this Cauchy problem is $\frac{p}{q-ps_1}$-summable in the direction $d$ and for every $\varepsilon>0$ there exists $r>0$ such that
its $\frac{p}{q-ps_1}$-sum $u\in\OO(S_{d}(\frac{\pi(q-ps_1)}{p}-{\varepsilon},r)\times D)$ is given by 
\begin{multline*}
u(t,z)=u^{d}(t,z)\\=\frac{1}{q}\*\int_{e^{i d}\RR_+}e_m(s/t)\*\*\biggl(\varphi(z+\sqrt[q]{s^p})+\varphi(z+e^{\frac{2\pi\*i}{q}}\*\sqrt[q]{s^{p}})+\dots+\varphi(z+e^{\frac{2(q-1)\pi\*i}{q}}\*\sqrt[q]{s^p})\biggr)\frac{ds}{s}.
\end{multline*}

Let $\theta:=\arg z_0$, $\delta:=\frac{q\theta}{p}$. 
Then  $\mathcal{L}_{\delta+\frac{2\pi j}{p}}$   ($j=0,1,\ldots,p-1$) are Stokes lines for $\widehat{u}$.
For every sufficiently small $\varepsilon>0$ there exists $r>0$ such that for every fixed $z\in D_r$ the jump is given by 
\begin{multline*}
J_{\mathcal{L}_{\delta}}\widehat u(t,z)= u^{\delta+\varepsilon}(t,z) - u^{\delta-\varepsilon}(t,z)=F_z(s)\bigg[\frac{e_m(s/t)}{s}\bigg]\\
=\biggl[\varphi(z+\sqrt[q]{s^p})+\varphi(z+e^{\frac{2\pi\*i}{q}}\*\sqrt[q]{s^{p}})+\dots+\varphi(z+e^{\frac{2(q-1)\pi\*i}{q}}
\*\sqrt[q]{s^p})\biggr]_{\frac{q\theta_z}{p}}\biggl[\frac{e_m(s/t)}{qs}\biggr]\\
=\biggl[\varphi(z+\sqrt[q]{s^p})\biggr]_{\frac{q\theta_z}{p}}\biggl[\frac{e_m(s/t)}{qs}\biggr]
\end{multline*}
with $\theta_z=\arg(z_0-z)$.
The last equality arising from the fact that in this case all singular points appear in the function $s\mapsto\varphi(z+\sqrt[q]{s^p})$.

Observe that from \cite[Theorem 32]{B2} one can derive the function
$$e_m(u)=T^{-}_{m_1,d}\,\bigg(e_{m_2}\big(1/z\big)\bigg)(1/u)=-\frac{1}{2\pi i}\int_{\gamma(d)}E_{m_1}\bigg(\frac{1}{uz}\bigg)\frac{p}{q}\bigg(\frac{1}{z}\bigg)^{\frac{p}{q}}e^{-\big(\frac{1}{z}\big)^{\frac{p}{q}}}\frac{dz}{z},
$$
where $E_{m_1}\big(\frac{1}{uz}\big)=\sum_{n=0}^\infty \frac{\big(\frac{1}{uz}\big)^n}{m_1(n)}$, $m_2(n)=\Gamma(1+\frac{q}{p}n)$ and,
by Example \ref{ex:functions}, $e_{m_2}(z)=\frac{p}{q}z^{\frac{p}{q}}e^{-z^{\frac{p}{q}}}$.
\bigskip

\emph{Case 2.}
Let us now study the formal solution $$\widehat{u}(t,z)=\sum_{n=0}^{\infty}\frac{\partial_{m_2,z}^n\varphi (z)}{m_1(n)}t^{n}$$ of the following equation
$$\left\{ \begin{array}{ll} \partial_{m_1,t}u(t,z)=\partial_{m_2,z}u(t,z),\\
u(0,z)=\varphi(z)
\end{array}\right.$$
with $\varphi(z)\in\OO^{\frac{1}{s_2-s_1}}\Bigl(\widetilde{\CC\setminus\{z_0\}}\Bigr)$ for some $z_0\in\CC\setminus\{0\}$; where $m_1$ is a moment function of order $s_1>0$ corresponding to a kernel function $e_{m_1}(z)$ of order $1/s_1$, $m_2$ is a moment function of order $s_2>0$ corresponding to a kernel function $e_{m_2}(z)$ of order $1/s_2 $ and $s_2>s_1$.\\
First, we apply to $\widehat u(t,z)$ the $m$-moment Borel transform
$$(\Bo_m\widehat{u})(t,z)=\Bo_m\bigg(\sum_{n=0}^{\infty}\frac{\partial_{m_2,z}^n\varphi (z)}{m_1(n)}\* t^{n}\bigg)= 
\sum_{n=0}^{\infty}\frac{\partial_{m_2,z}^n\varphi (z)}{m_2(n)}\* t^{n},$$
where $m(n):=m_2(n)/m_1(n)$ is a moment function of order $s_2-s_1$ corresponding to a kernel function $e_m(z)$ of order $k:=\frac{1}{s_2-s_1}$.

Using \cite[Proposition 3]{Mic7} we see that for $|z|<\varepsilon<r$ and $n\in\NN$ we have
$$\partial_{m_2,z}^n\varphi (z)=\frac{1}{2\pi i}\oint_{|w|=\varepsilon}\varphi(w)\int_0^{\infty(\psi)}\zeta^n E_{m_2}(z\zeta)\frac{e_{m_2}(w\zeta)}{w\zeta}d\zeta dw,$$
where $\psi\in(-\arg w -\frac{\pi s_2}{2},-\arg w+\frac{\pi s_2}{2}).$
Thus
\begin{multline*}(\Bo_m\widehat{u})(t,z)=\sum_{n=0}^{\infty}\frac{\partial_{m_2,z}^n \varphi (z)\* t^{n}}{m_2(n)}=
\sum_{n=0}^{\infty}\frac{1}{2\pi i}\oint_{|w|=\varepsilon}\varphi(w)\int_0^{\infty(\psi)}\frac{\zeta^n t^n}{m_2(n)}E_{m_2}(z\zeta)
\frac{e_{m_2}(w\zeta)}{w\zeta}d\zeta dw\\
= \frac{1}{2\pi i}\oint_{|w|=\varepsilon}\varphi(w)\int_0^{\infty(\psi)}E_{m_2}(t\zeta)E_{m_2}(z\zeta)\frac{e_{m_2}(w\zeta)}{w\zeta}d\zeta dw.
\end{multline*}

Let $f(s,z):=(\Bo_m\widehat{u})(s,z)$, then by using $m$-moment Laplace transform in a nonsingular direction $d$ we get 
\begin{multline*} (T_{m,d}f)(t,z)=\int_{e^{id}\RR_+}e_m(s/t)\*f(s,z)\frac{ds}{s}\\
 =\int_{e^{id}\RR_+}e_m(s/t)\bigg(\frac{1}{2\pi i}\oint_{|w|=\varepsilon}\varphi(w)\int_0^{\infty(\psi)}E_{m_2}(s\zeta)E_{m_2}(z\zeta)\frac{e_{m_2}(w\zeta)}{w\zeta}d\zeta dw\bigg) \frac{ds}{s}.
 \end{multline*}
Notice that, by \cite[Theorem 32]{B2}, the function $e_m(u)$ is of the form
\begin{equation}
\label{eq: c2 e_m}
e_m(u)=T^{-}_{m_1,d}\,\bigg(e_{m_2}\big(1/z\big)\bigg)(1/u)=-\frac{1}{2\pi i}\int_{\gamma(d)}E_{m_1}\bigg(\frac{1}{uz}\bigg)e_{m_2}(1/z)\frac{dz}{z},
\end{equation}
where $E_{m_1}\big(\frac{1}{uz}\big)=\sum_{n=0}^\infty \frac{\big(\frac{1}{uz}\big)^n}{m_1(n)}$.

Thus, by Proposition \ref{pr:sum}, the unique formal solution $\widehat{u}(t,z)$ of this Cauchy problem is $k$-summable in the direction $d$
and for every $\varepsilon>0$ there exists $r>0$ such that
its $k$-sum $u\in\OO(S_{d}(\frac{\pi}{k}-{\varepsilon},r)\times D)$ is given by 
\begin{multline*}
u(t,z)=u^{ d}(t,z)\\=\int_{e^{i d}\RR_+}e_m(s/t)\bigg(\frac{1}{2\pi i}\oint_{|w|=\varepsilon}\varphi(w)\int_0^{\infty(\psi)}E_{m_2}(s\zeta)E_{m_2}(z\zeta)\frac{e_{m_2}(w\zeta)}{w\zeta}d\zeta dw\bigg)\frac{ds}{s}.
\end{multline*}

Then $\mathcal{L}_{\delta}$, with $\delta=\theta:=\arg z_0$, is a Stokes line for $\widehat{u}$. For $z=0$ the jump is given by 
\begin{multline*}
J_{\mathcal{L}_{\delta}}\widehat u(t,0)= u^{\delta^+}(t,0) - u^{\delta^-}(t,0)=F_0(s)\bigg[\frac{e_m(s/t)}{s}\bigg]\\
=\biggl[\frac{1}{2\pi i}\oint_{|w|=\varepsilon}\varphi(w)\int_0^{\infty(\psi)}E_{m_2}(s\zeta)\frac{e_{m_2}(w\zeta)}{w\zeta}d\zeta dw\biggr]_{\theta}\biggl[\frac{e_m(s/t)}{s}\biggr].
\end{multline*}
Using \cite[formula (5.15)]{B2} one can derive
$$\int_0^{\infty(\psi)}E_{m_2}(s\zeta)\frac{e_{m_2}(w\zeta)}{w\zeta}d\zeta=\frac{1}{w-s},$$
hence 
$$ J_{\mathcal{L}_{\delta}}\widehat u(t,0)=\biggl[\frac{1}{2\pi i}\oint_{|w|=\varepsilon}\frac{\varphi(w)}{w-s} dw\biggr]_{\theta}\biggl[\frac{e_m(s/t)}{s}\biggr] =\bigg[\varphi(s)\bigg]_{\theta}\biggl[\frac{e_m(s/t)}{s}\biggr],$$
where the last equality follows from the Cauchy integral formula.
\bigskip

\emph{Case 3.}
Now, we take the following equation under consideration
$$\left\{ \begin{array}{ll} \partial_{m_1,t}^q u(t,z)=\partial_{m_2,z}^q u(t,z),\\
u(0,z)=\varphi(z),\\
\partial_{m_1,t}^j u(0,z)=0,\,\, \mathrm{for}\,\, j-1,2,\dots,q-1
\end{array}\right.$$
with $\varphi(z)\in\OO^{\frac{1}{s_2-s_1}}(\widetilde{\CC\setminus\{z_0\}})$ for some $z_0\in\CC\setminus\{0\}$; where $m_1$ is a moment function of order $s_1>0$ corresponding to a kernel function $e_{m_1}(z)$ of order $1/s_1$, $m_2$ is a moment function of order $s_2>0$ corresponding to a kernel function $e_{m_2}(z)$ of order $1/s_2 $ and $s_2>s_1$.\\
 Observe that since 
 $$\partial_{m_1,t}^q-\partial_{m_2,z}^q=(\partial_{m_1,t}-\partial_{m_2,z})(\partial_{m_1,t}-e^{\frac{2\pi i}{q}}\partial_{m_2,z})\cdot\ldots\cdot(\partial_{m_1,t}-e^{\frac{2\pi i(q-1)}{q}}\partial_{m_2,z}),$$
then we can write
$$\widehat{u}(t,z)=\widehat{u}_0(t,z)+\widehat{u}_1(t,z)+\ldots+\widehat{u}_{q-1}(t,z),$$
where, for $j=0,1,\ldots,q-1$
$$\widehat{u}_j(t,z)=\frac{1}{q}\sum_{n=0}^{\infty}\frac{\partial_{m_2,z}^{n}\varphi (z)}{m_1(n)}\big(e^{\frac{2\pi ij}{q}}\big)^n t^{n},$$
is a formal solution of the equation
$$\left\{ \begin{array}{ll} \partial_{m_1,t} u_j(t,z)=e^{\frac{2\pi ij}{q}}\partial_{m_2,z} u_j(t,z),\\
u_j(0,z)=\frac{1}{q}\varphi(z)\in\OO^{\frac{1}{s_2-s_1}}(\widetilde{\CC\setminus\{z_0\}}).
\end{array}\right.$$
Notice that, based on reasoning of the case 2, for each $\widehat u_j(t,z)$ we obtain that
$$
u_j(t,z)=\int_{e^{i d}\RR_+}e_m(s/t)\bigg(\frac{1}{2\pi i}\oint_{|w|=\varepsilon}\varphi(w)
\int_0^{\infty(\psi)}E_{m_2}(se^{\frac{2\pi ij}{q}} \zeta)E_{m_2}(z\zeta)\frac{e_{m_2}(w\zeta)}{w\zeta}d\zeta dw\bigg)\frac{ds}{s}.$$
So $\mathcal{L}_{\delta+\frac{2\pi j}{q}}$ are Stokes lines for $\widehat{u}(t,z)$, where $\delta=\theta=\arg z_0$ and $j=0,\dots,q-1$.
Moreover
\begin{multline*} J_{\mathcal{L}_{\delta}}\widehat u_j(t,0)=\biggl[\frac{1}{2\pi i}\oint_{|w|=\varepsilon}\frac{\varphi(w)}{w-se^{\frac{2\pi ij}{q}}} dw\biggr]_{\theta}\biggl[\frac{e_m(s/t)}{qs}\biggr] =\bigg[\varphi\big(se^{\frac{2\pi ij}{q}}\big)\bigg]_{\theta}\biggl[\frac{e_m(s/t)}{qs}\biggr]\\
=\begin{cases} \bigg[\varphi(s)\bigg]_{\theta}\biggl[\frac{e_m(s/t)}{qs}\biggr],\,\, \mathrm{for}\,\,j=0\\
0, \,\,\mathrm{for}\,\,j=1,2,\ldots,q-1,
\end{cases}
\end{multline*}
where $e_m$ is given by (\ref{eq: c2 e_m}). Thus 
$$J_{\mathcal{L}_{\delta}}\widehat u(t,0)=\bigg[\sum_{j=0}^{q-1}\varphi\big(se^{\frac{2\pi ij}{q}}\big)\bigg]_{\theta}\biggl[\frac{e_m(s/t)}{qs}\biggr]=\bigg[\varphi(s)\bigg]_{\theta}\biggl[\frac{e_m(s/t)}{qs}\biggr].
$$
\bigskip

\emph{Case 4.} In this part we will study more general case i.e.
\begin{equation}
\label{eq:18}
\left\{ \begin{array}{ll} \partial_{m_1,t}^{p}u(t,z)=\partial_{m_2,z}^{q}u(t,z), \,\,\, \mathrm{with}\,\, 0< ps_1<qs_2,\\
u(0,z)=\varphi(z),\,\, \\ 
\partial_{m_1,t}^{j}u(0,z)=0,\,\, \textrm{for}\,\,\, j=1,2,\dots,p-1
\end{array}\right.
\end{equation}
with $\varphi(z)\in\OO^{\frac{q}{qs_2-ps_1}}\Bigl(\widetilde{\CC\setminus\{z_0\}}\Bigr)$ for some $z_0\in\CC\setminus\{0\}$; where $m_1$ is a moment function of order $s_1>0$ corresponding to a kernel function $e_{m_1}(z)$ of order $1/s_1,$ $m_2$ is a moment function of order $s_2>0$ corresponding to a kernel function $e_{m_2}(z)$ of order $1/s_2.$\\
The above Cauchy problem has a formal solution 
$$\widehat{u}(t,z)=\sum_{n=0}^{\infty}\frac{\partial_{m_2,z}^{qn}\varphi(z)}{m_1(pn)}\* t^{pn}.$$
By Theorem \ref{th:summability}, $\mathcal{L}_{\delta+\frac{2\pi j}{p}}$ are Stokes lines for $\widehat{u}(t,z)$, where $\delta=\frac{q}{p}\theta$,
$\theta:=\arg z_0$ and $j=0,\dots,p-1$. To calculate the jumps across Stokes lines
assume that $v(t,z):=u(t^{\frac{q}{p}},z)$ and $\tilde m_1(n):=m_1\big(\frac{pn}{q}\big)$ is a moment function of order $\frac{s_1p}{q}>0$
corresponding
to a kernel function $e_{\tilde m_1}(z)=\frac{q}{p} e_{m_1}\big(z^{\frac{q}{p}}\big)$ of order $\frac{q}{s_1p}$.
Then $$\widehat v(t,z)=\sum_{n=0}^{\infty}\frac{\partial_{m_2,z}^{qn}\varphi(z)}{\tilde m_1(qn)}\* t^{qn}$$
is a formal solution of the equation 
\begin{equation}
\label{eq:19}
\left\{ \begin{array}{ll} \partial_{\tilde m_1,t}^{q}v(t,z)=\partial_{m_2,z}^{q}v(t,z),\\
v(0,z)=\varphi(z)\in\OO^{\frac{q}{qs_2-ps_1}}(\widetilde{\CC\setminus\{z_0\}}), \\ 
\partial_{\tilde m_1,t}^{j}v(0,z)=0,\,\, \textrm{for}\,\,\, j=1,2,\dots,q-1.
\end{array}\right.
\end{equation}
Observe that $\widehat u(t,z)$ is a formal solution of the equation (\ref{eq:18}) if and only if $\widehat v(t,z)$ is a formal solution of the equation (\ref{eq:19}) (see also \cite[Lemma 3]{Mic7}).

In this case, we reduce our problem to the one we considered in the case 3. Thus, based on the obtained results we have 
$$v_j(t,z)=\int_{e^{i d}\RR_+}e_m(s/t)\bigg(\frac{1}{2\pi i}\oint_{|w|=\varepsilon}\varphi(w)\int_0^{\infty(\psi)}E_{m_2}(se^{\frac{2\pi ij}{q}} \zeta)E_{m_2}(z\zeta)\frac{e_{m_2}(w\zeta)}{w\zeta}d\zeta dw\bigg)\frac{ds}{s},$$
so 
\begin{multline*} J_{\mathcal{L}_{\theta}}\widehat v_j(t,0)=\biggl[\frac{1}{2\pi i}\oint_{|w|=\varepsilon}\frac{\varphi(w)}{w-se^{\frac{2\pi ij}{q}}} dw\biggr]_{\theta}\biggl[\frac{e_m(s/t)}{qs}\biggr] =\bigg[\varphi\big(se^{\frac{2\pi ij}{q}}\big)\bigg]_{\theta}\biggl[\frac{e_m(s/t)}{qs}\biggr]\\
=\begin{cases} \bigg[\varphi(s)\bigg]_{\theta}\biggl[\frac{e_m(s/t)}{qs}\biggr],\,\, \mathrm{for}\,\,j=0\\
0, \,\,\mathrm{for}\,\,j=1,2,\ldots,q-1,
\end{cases}
\end{multline*}
where $e_m$ is given by 
$$e_m(u)=T^{-}_{\tilde{m}_1,d}\,\bigg(e_{m_2}\big(1/z\big)\bigg)(1/u)=-\frac{1}{2\pi i}\int_{\gamma(d)}E_{\tilde{m}_1}\bigg(\frac{1}{uz}\bigg)e_{m_2}(1/z)\frac{dz}{z},$$
with $$E_{\tilde m_1}\bigg(\frac{1}{uz}\bigg)=\sum_{n=0}^\infty \frac{\big(\frac{1}{uz}\big)^n}{\tilde m_1(n)}=\sum_{n=0}^\infty \frac{\big(\frac{1}{uz}\big)^n}{{m_1}(\frac{pn}{q})}.$$
Thus 
$$J_{\mathcal{L}_{\theta}}\widehat v(t,0)=\bigg[\sum_{j=0}^{q-1}\varphi\big(se^{\frac{2\pi ij}{q}}\big)\bigg]_{\theta}\biggl[\frac{e_m(s/t)}{qs}\biggr]=\bigg[\varphi(s)\bigg]_{\theta}\biggl[\frac{e_m(s/t)}{qs}\biggr].
$$
Hence $$J_{\mathcal{L}_{\delta}}\widehat u(t,0)=J_{\mathcal{L}_{\theta}}\widehat v(t^{\frac{p}{q}},0)=\bigg[\varphi(s)\bigg]_{\theta}\biggl[\frac{e_m(s/t^{\frac{p}{q}})}{qs}\biggr].$$

\bibliographystyle{siam}
\bibliography{summa}
\end{document}